\documentclass[reqno,a4paper, 11pt]{amsart}

\usepackage[a4paper=true,pdfpagelabels]{hyperref}
\usepackage{graphicx}

\usepackage[ansinew]{inputenc}
\usepackage{amsfonts,epsfig}
\usepackage{latexsym}
\usepackage{amsmath}
\usepackage{amssymb}
\usepackage{mathabx}

\newtheorem{theorem}{Theorem}
\newtheorem{lemma}[theorem]{Lemma}

\newtheorem{proposition}[theorem]{Proposition}

\newtheorem{lettertheorem}{Theorem}
\newtheorem{letterlemma}[lettertheorem]{Lemma}

\theoremstyle{definition}

\theoremstyle{remark}

\numberwithin{equation}{section}

\newcommand{\D}{\mathbb{D}}
\newcommand{\DD}{\widehat{\mathcal{D}}}
\newcommand{\Dd}{\widecheck{\mathcal{D}}}
\newcommand{\DDD}{\mathcal{D}}
\newcommand{\N}{\mathbb{N}}

\newcommand{\C}{\mathbb{C}}
\newcommand{\fgrnu}{\widehat{f}_{r,\nu}}

\renewcommand{\phi}{\varphi}

\def\BMO{\mathord{\rm BMO}}
\def\MO{\mathord{\rm MO}}
\def\BO{\mathord{\rm BO}}
\def\BA{\mathord{\rm BA}}
\newcommand{\Bg}{\mathcal{B}_{d\gamma}}

\def\a{\alpha}       \def\b{\beta}        \def\g{\gamma}
\def\d{\delta}           
     \def\om{\omega}      
\def\s{\sigma}              
         \def\r{\rho}         \def\z{\zeta}
                  \def\vp{\varphi}
\def\G{\Gamma}

\renewcommand{\H}{\mathcal{H}}

\newenvironment{Prf}{\noindent{\emph{Proof of}}}
{\hfill$\Box$ }

\begin{document}

\title[Hankel operators induced by radial Bekoll\'e-Bonami weights]{Hankel operators induced by radial Bekoll\'e-Bonami weights on  Bergman spaces}


\thanks{This research was supported in part by Ministerio de Econom\'{\i}a y Competitivivad, Spain, projects
MTM2014-52865-P and MTM2015-69323-REDT; La Junta de Andaluc{\'i}a,
project FQM210; Academy of Finland project no. 268009. AP acknowledges financial support from the Spanish Ministry of Economy and Competitiveness, through the Mar\'ia de Maeztu Programme for Units of Excellence in R\&D (MDM-2014-0445), Academy of Finland project no. 268009, and the grant MTM2017-83499-P (Ministerio de Educaci\'on y Ciencia).}

\author{Jos\'e \'Angel Pel\'aez}
\address{Departamento de An\'alisis Matem\'atico, Universidad de M\'alaga, Campus de
Teatinos, 29071 M\'alaga, Spain} \email{japelaez@uma.es}

\author[Antti Per\"{a}l\"{a}]{Antti Per\"{a}l\"{a}}
\address{Antti Per\"{a}l\"{a} \\Departament de Matem\`{a}tiques i Inform\`{a}tica,
Universitat de Barcelona, 08007 Barcelona, Catalonia, Spain,
Barcelona Graduate School of Mathematics (BGSMath).} \email{perala@ub.edu}

\author[Jouni R\"atty\"a]{Jouni R\"atty\"a}
\address{University of Eastern Finland\\
Department of Physics and Mathematics\\
P.O.Box 111\\FI-80101 Joensuu\\
Finland}
\email{jouni.rattya@uef.fi}


\date{\today}
\maketitle

\begin{abstract}
We study big Hankel operators $H_f^\nu:A^p_\omega \to L^q_\nu$ generated by radial Bekoll\'e-Bonami weights $\nu$, when $1<p\leq q<\infty$. Here the radial weight $\omega$ is assumed to satisfy a two-sided doubling condition, and $A^p_\omega$ denotes the corresponding weighted Bergman space. A characterization for simultaneous boundedness of $H_f^\nu$ and $H_{\overline{f}}^\nu$ is provided in terms of a general weighted mean oscillation. Compared to the case of standard weights that was recently obtained by Pau, Zhao and Zhu (Indiana Univ. Math. J. 2016), the respective spaces depend on the weights $\omega$ and $\nu$ in an essentially stronger sense. This makes our analysis deviate from the blueprint of this more classical setting. As a consequence of our main result, we also study the case of anti-analytic symbols.
\end{abstract}

\section{Introduction and main results}

Let $\H(\D)$ denote the space of analytic functions in the unit disc $\D=\{z\in\C:|z|<1\}$.
A function $\om:\D\to[0,\infty)$, integrable over the unit disc $\D$, is called a weight. It  is radial if $\om(z) = \om(|z|)$ for all $z\in\D$.
 For $0<p<\infty$ and a weight $\om$, the Lebesgue space $L^p_\om$ consists of (equivalence classes of) complex-valued measurable functions $f$ in $\D$ such that
    $$
    \|f\|_{L^p_\om}=\left(\int_\D|f(z)|^p\om(z)\,dA(z)\right)^{\frac1p}<\infty,
    $$
where $dA(z) = dx\,dy/\pi$ denotes the normalized Lebesgue area measure on $\D$. The weighted Bergman space $A^p_\om$ is the space of analytic functions in $L^p_\om$. As usual, $A^p_\alpha$ denotes the weighted Bergman
space induced by the standard radial weight $(\alpha+1)(1-|z|^2)^\alpha$.
 If $\nu$ is a radial weight then $A^2_\nu$ is a closed subspace of $L^2_\nu$ and  the orthogonal projection from $L^2_\nu$ to $A^2_\nu$ is given by
    $$
    P_\nu(f)(z) = \int_\D f(\zeta)\overline{B^\nu_z(\zeta)}\nu(\zeta)\,dA(\zeta),\quad z\in\D,
    $$
where $B^\nu_z$ are the reproducing kernels of $A^2_\nu$; $f(z)=\langle f, B^\nu_z\rangle_{A^2_\nu}$ for all $z\in\D$ and $f\in A^2_\nu$.

The study of the boundedness of weighted Bergman projections on $L^p$-spaces is a compelling topic that has attracted a considerable amount of attention during the last decades. A  well known result due to Bekoll\'e and Bonami~\cite{Bek,BB} describes the weights $\om$ such that the Bergman projection $P_\eta$, induced by the standard weight $(\eta+1)(1-|z|^2)^\eta$, is bounded on $L^q_\om$ for $1<q<\infty$. We denote this class of weights by $B_q(\eta)$, and write $B_q=\cup_{\eta>-1}B_q(\eta)$ for short. In the case of a standard weight, the Bergman reproducing kernels are given by the neat formula $(1-\overline{z}\zeta)^{-(2+\eta)}$. However, for a general radial weight $\nu$ the Bergman reproducing kernels $B^\nu_z$ may have zeros \cite{Peralakernels} and such explicit formulas for the kernels do not necessarily exist. This is one of the main obstacles in dealing with $P_\nu$~\cite{CP2,PR2016/1}. Nonetheless, we shall prove in Proposition~\ref{pr:bqDDD} below that if $\nu\in B_q$ is radial, then $P_\nu:L^q_\nu\to L^q_\nu$ is bounded for each $1<q<\infty$. The proof of this relies on accurate estimates for the integral means of $B^z_\nu$ recently obtained in \cite[Theorem~1]{PR2016/1}, and the result itself plays an important  role in the proof of the main discovery of this paper.

All the above makes the class of radial weights in $B_q$ an appropriate framework for the study of the big Hankel operator
    $$
    H_f^\nu(g)(z)=(I-P_\nu)(fg)(z),\quad f\in L^1_\nu,\quad z\in\D,
    $$
on weighted Bergman spaces. For an analytic function $f$, the Hankel operator $H^\b_{\overline{f}}$, induced by a standard projection, has been widely studied on Bergman spaces since the pioneering work of Axler~\cite{AxlerDuke}, which was later extended in \cite{AFP88}. In the case of a rapidly decreasing weight $\nu$ and $f\in\H(\D)$, Galanopoulos and Pau~\cite{GalPauAASF} did an extensive research on $H^\nu_{\overline{f}}$ on $A^2_\nu$;    see~\cite{Arrousi} for further results. For general symbols, Zhu~\cite{ZhuTAMS87} was the first to build up a bridge between Hankel operators and the mean oscillation of the symbols in the Bergman metric, and this idea has been further developed in several contexts~\cite{BBCZ,BCZ1,BCZ2,Zhu1992}; see \cite{Zhu} and the references therein for further information on the theory of Hankel operators. More recently, Pau, Zhao and Zhu~\cite{PZZ} described
the complex valued symbols $f$ such that the Hankel operators $H_f^\b$ and $H^\b_{\overline{f}}$ are simultaneously bounded from $A^p_\alpha$ to $L^q_\b$, provided $1<p\le q<\infty$. Our primary aim is to extend this last-mentioned result to the context of radial $B_q$-weights. To do this, some definitions are needed. For a radial weight $\om$, we assume throughout the paper that $\widehat{\om}(z)=\int_{|z|}^1 \om(s)\,ds>0$ for all $z \in \D$, for otherwise the Bergman space $A^p_\om$ would contain all analytic functions in $\D$. A radial weight $\om$ belongs to the class~$\DD$ if there exists a constant $C=C(\om)>1$ such that $\widehat{\om}(r)\le C\widehat{\om}(\frac{1+r}{2})$ for all $0\le r<1$. Moreover, if there exist $K=K(\om)>1$ and $C=C(\om)>1$ such that
    \begin{equation}\label{K}
    \widehat{\om}(r)\ge C\widehat{\om}\left(1-\frac{1-r}{K}\right),\quad 0\le r<1,
    \end{equation}
then $\om\in\Dd$. We write $\DDD=\DD\cap\Dd$ for short. For basic properties of these classes of weights and more, see~\cite{PelSum14,PR2014Memoirs} and the references therein. Let $\b(z,\z)$ denote the hyperbolic distance between $z,\z\in\D$, $\Delta(z,r)$ the hyperbolic disc of center $z$ and radius $r>0$, and $S(z)$ the Carleson square associated to $z$. For $0<p,q<\infty$ and radial weights $\om,\nu$, define
    \begin{equation}\label{eq:gammaintro}
    \g(z)=\g_{\om,\nu,p,q}(z)=\frac{\widehat{\nu}(z)^\frac1q(1-|z|)^{\frac1q}}{\widehat{\om}(z)^\frac1p(1-|z|)^{\frac1p}},\quad z\in\D.
    \end{equation}
Further, for $f\in L^1_{\nu,{\rm loc}}$, write $\fgrnu(z)=\frac{\int_{\Delta(z,r)}f(\z)\nu(\z)\,dA(\z)}{\nu(\Delta(z,r))}$ and
    $$
    \MO_{\nu,q,r}(f)(z)=\left(\frac{1}{\nu(\Delta(z,r))} \int_{\Delta(z,r)}|f(\z)-\fgrnu(z)|^q\nu(\z)\,dA(\z)\right)^{\frac{1}{q}}
    $$
for all $z\in\D$. It is worth noticing that for prefixed $r>0$, the quantity $\nu(\Delta(z,r))$ may equal to zero for some $z$ arbitrarily close to the boundary if $\nu\in\DD$. However, if $\nu\in\DDD$, then there exists $r_0=r_0(\nu)>0$ such that $\nu(\Delta(z,r))\asymp\nu(S(z))>0$ for all $z\in\D$ if $r\ge r_0$. The space $\BMO(\Delta)_{\om,\nu,p,q,r}$ consists of $f\in L^q_{\nu,{\rm loc}}$ such that
    $$
    \|f\|_{\BMO(\Delta)_{\om,\nu,p,q,r}}=\sup_{z\in\D}\left(\MO_{\nu,q,r}(f)(z)\g(z)\right)<\infty.
    $$
We will show that if $\nu\in\DDD$, then $\BMO(\Delta)_{\om,\nu,p,q,r}$ does not depend on $r$ for $r\ge r_0$. In this case, we simply write $\BMO(\Delta)_{\om,\nu,p,q}$. The main result of this study reads as follows and it will be proved in Section~\ref{Proof of main theorem}.

\begin{theorem}\label{th:Hankelqbiggerp}
Let $1<p\le q<\infty$, $\om\in\DDD$, $\nu\in B_q$ a radial weight and $f\in L^q_\nu$.
Then  $H_f^\nu,H_{\overline{f}}^\nu:A^p_\om\to L^q_\nu$ are bounded if and only if $f\in \BMO(\Delta)_{\om,\nu,p,q}$.
\end{theorem}

The approach employed  in the proof of this result follows the guideline of  \cite[Thorem~4.1]{PZZ}, however a good number of steps
cannot be adapted straightforwardly and need substantial modifications. In Section~\ref{Auxiliary results} we prove some results concerning the classes of weights
involved in this work and the boundedness of the Bergman projection $P_\nu$, while in Section~\ref{Some spaces of functions} we introduce and study two spaces of functions on $\D$.
One of them is denoted as $\BA(\Delta)_{\om,\nu,p,q}$, and although its initial definition depends on $r$, it can be described in terms of an appropriate Berezin transform or simply observing that $f\in\BA(\Delta)_{\om,\nu,p,q}$ if and only the multiplication operator $M_f(g)=fg$ is bounded from $A^p_\om$ to $L^q_\nu$~\cite{PR2015/2}. The second one, denoted by $\BO(\Delta)_{\om,\nu,p,q}$, consists of continuous functions on $\D$ such that the oscillation in the Bergman metric is bounded in terms of the auxiliary function $\g$ given in \eqref{eq:gammaintro}. We show that $f\in \BO(\Delta)_{\om,\nu,p,q}$ if and only if
    $$
    |f(z)-f(\z)|\lesssim\|f\|_{\BO(\Delta)_{\om,\nu,p,q}}(1+\beta(z,\z))\Gamma_\tau(z,\z)\quad z,\z\in\D,
    $$
where
    \begin{equation*}
    \Gamma_\tau(z,\z)=\frac{\left(\frac{|1-\overline{z}\z|^2}{\max\{1-|z|^2,1-|\z|^2\}}\right)^{\frac{1}{p}-\frac{\tau+1}{q}}
    \widehat{\om}\left(1-\frac{2|1-\overline{z}\z|^2}{\max\{1-|z|^2,1-|\z|^2\}}\right)^{\frac{1}{p}}}
    {\min\left\{ \frac{\widehat{\nu}(z)}{(1-|z|)^\tau},
    \frac{\widehat{\nu}(\z)}{(1-|\z|)^\tau}\right\}^{\frac1q}},\quad z,\z\in\D,
    \end{equation*}
for an appropriate (small) constant $\tau=\tau(\om,\nu)>0$. If $\om$ and $\nu$ are standard weights, then $\Gamma_\tau$ does not coincide with the function playing the corresponding role in \cite[Lemma~3.2]{PZZ}; in the latter case the function is simpler in many aspects and does not depend on the additional parameter $\tau$. Then, we show that
    \begin{equation}\label{eq:deco}
    \BMO(\Delta)_{\om,\nu,p,q}=\BA(\Delta)_{\om,\nu,p,q}+\BO(\Delta)_{\om,\nu,p,q}.
    \end{equation}
In order to prove this decomposition, due to the complex nature of $ \Gamma_\tau(z,\z)$, we are forced to split $\D$ into several regions depending on $z$, establish sharp estimates for $\Gamma_\tau(z,\z)$ in each region and then apply properties of weights in $\DDD$. The identity \eqref{eq:deco} together with a description of the boundedness of the integral operator $$T_{b,c}f(z)=\int_\D f(\z)\frac{(1-|\z|^2)^b}{(1-z\overline{\z})^c}\,dA(\z)$$ and its maximal counterpart from $A^p_\om$ to $L^q_\nu$, see Section~\ref{Boundedness of integral operators} below,
are  key tools to prove that each $f\in \BMO(\Delta)_{\om,\nu,p,q}$ induces a bounded Hankel operator from $A^p_\om$ to $L^q_\nu$. Theorem~\ref{th:Hankelqbiggerp} will be proved in Section~\ref{Proof of main theorem}.

Finally, in Section~\ref{Anti-analytic symbols}, as a byproduct of Theorem~\ref{th:Hankelqbiggerp}, we describe the analytic symbols such that $H_{\overline{f}}:A^p_\om\to L^q_\nu$ is bounded. The space $\Bg$ consists of $f\in\H(\D)$ such that
    $$
    \|f\|_{\Bg}=\sup_{z\in\D}|f'(z)|(1-|z|)\gamma(z)+|f(0)|<\infty,
    $$
where $\gamma$ is given by \eqref{eq:gammaintro}.

\begin{theorem}\label{th:analytic}
Let $1<p\le q<\infty$, $\om\in\DDD$, $\nu\in B_q$ a radial weight  and $f\in A^1_\nu$.
Then $H^{\nu}_{\overline{f}}:A^p_\om \to L^q_\nu$ is bounded if and only if $f\in\Bg$.
\end{theorem}

Throughout the paper $\frac{1}{p}+\frac{1}{p'}=1$ for $1<p<\infty$. Further, the letter $C=C(\cdot)$ will denote an
absolute constant whose value depends on the parameters indicated
in the parenthesis, and may change from one occurrence to another.
We will use the notation $a\lesssim b$ if there exists a constant
$C=C(\cdot)>0$ such that $a\le Cb$, and $a\gtrsim b$ is understood
in an analogous manner. In particular, if $a\lesssim b$ and
$a\gtrsim b$, then we will write $a\asymp b$.

\section{Auxiliary results}
\label{Auxiliary results}

For a radial weight $\om$, $K>1$ and $0\le r<1$, let $\r_n^r=\r_n^r(\om,K)$ be defined by
$\widehat{\om}(\r^r_n)=\widehat{\om}(r)K^{-n}$ for all $n\in\N\cup\{0\}$. Write $\r_n=\r^0_n$ for short. For $x>-1$, write $\om_x=\int_0^1r^x\om(r)\,dr$. Denote
    $$
    \omega^\star(z)=\int_{|z|}^1\log\frac{s}{|z|}\om(s)s\,ds,\quad z\in\D\setminus\{0\}.
    $$

Throughout the proofs we will repeatedly use several basic properties of weights in the classes $\DD$ and $\Dd$. For a proof of the first lemma, see~\cite[Lemma~2.1]{PelSum14}; the second one can be proved by similar arguments.

\begin{letterlemma}\label{Lemma:weights-in-D-hat}
Let $\om$ be a radial weight. Then the following statements are equivalent:
\begin{itemize}
\item[\rm(i)] $\om\in\DD$;
\item[\rm(ii)] There exist $C=C(\om)>0$ and $\b=\b(\om)>0$ such that
    \begin{equation*}
    \begin{split}
    \widehat{\om}(r)\le C\left(\frac{1-r}{1-t}\right)^{\b}\widehat{\om}(t),\quad 0\le r\le t<1;
    \end{split}
    \end{equation*}
\item[\rm(iii)] There exist $C=C(\om)>0$ and $\gamma=\gamma(\om)>0$ such that
    \begin{equation*}
    \begin{split}
    \int_0^t\left(\frac{1-t}{1-s}\right)^\g\om(s)\,ds
    \le C\widehat{\om}(t),\quad 0\le t<1;
    \end{split}
    \end{equation*}
\item[\rm(iv)] There exists $\lambda=\lambda(\om)\ge0$ such that
    $$
    \int_\D\frac{dA(z)}{|1-\overline{\z}z|^{\lambda+1}}\asymp\frac{\widehat{\om}(\zeta)}{(1-|\z|)^\lambda},\quad \z\in\D;
    $$
\item[\rm(v)] There exist $K=K(\om)>1$ and $C=C(\om,K)>1$ such that $1-\r_n^r(\om,K)\ge C(1-\r^r_{n+1}(\om,K))$ for some (equivalently for all) $0\le r<1$ and for all $n\in\N\cup\{0\}$.
\end{itemize}
\end{letterlemma}

\begin{letterlemma}\label{Lemma:weights-in-D}
Let $\om$ be a radial weight. Then $\om\in\Dd$ if and only if there exist $C=C(\om)>0$ and $\a=\a(\om)>0$ such that
    \begin{equation*}
    \begin{split}
    \widehat{\om}(t)\le C\left(\frac{1-t}{1-r}\right)^{\a}\widehat{\om}(r),\quad 0\le r\le t<1.
    \end{split}
    \end{equation*}
\end{letterlemma}

Two more results on weights of more general nature than Lemmas~\ref{Lemma:weights-in-D-hat} and~\ref{Lemma:weights-in-D} are also needed.

\begin{lemma}\label{lemma:d-hat-new}
Let $\om$ be a radial weight. Then the following statements are equivalent:
\begin{itemize}
\item[\rm(i)] $\om\in\DD$;
\item[\rm(ii)] For some (equivalently for each) $\nu\in\DDD$ there exists a constant $C=C(\om,\nu)>0$ such that
    $$
    \int_r^1\frac{\om(t)\widehat{\nu}(t)}{\widehat{\om}(t)}\,dt\le C\widehat{\nu}(r),\quad 0\le r<1;
    $$
\item[\rm(iii)] For some (equivalently for each) $\nu\in\DDD$ there exists a constant $C=C(\om,\nu)>0$ such that
    $$
    \int_0^r\frac{\omega(t)}{\widehat{\omega}(t)\widehat{\nu}(t)}\,dt\le\frac{C}{\widehat{\nu}(r)},\quad 0\le r<1.
    $$
\end{itemize}
\end{lemma}

\begin{proof}
Let first $\om\in\DD$ and $0\le r<1$, and consider $\r_n^r=\r_n^r(\om,K)$ for all $n\in\N\cup\{0\}$. Then Lemma~\ref{Lemma:weights-in-D}, applied to $\nu\in\DDD\subset\Dd$, and Lemma~\ref{Lemma:weights-in-D-hat}(v), applied to $\om$, imply
    \begin{equation*}
    \begin{split}
    \int_r^1\frac{\om(t)\widehat{\nu}(t)}{\widehat{\om}(t)}\,dt
    &=\sum_{n=0}^\infty\int_{\r_n^r}^{\r_{n+1}^r}\frac{\om(t)\widehat{\nu}(t)}{\widehat{\om}(t)}\,dt
    \le\sum_{n=0}^\infty\widehat{\nu}(\r_n^r)\int_{\r_n^r}^{\r_{n+1}^r}\frac{\om(t)}{\widehat{\om}(t)}\,dt\\
    &\lesssim\log K\frac{\widehat{\nu}(\r_0^r)}{(1-\r_0^r)^\b}\sum_{n=0}^\infty(1-\r_n^r)^\b\\
    &\le\widehat{\nu}(r)\log K\sum_{n=0}^\infty\frac{1}{(C^\b)^n}=\widehat{\nu}(r)\log K\frac{C^\b}{C^\b-1},\quad0\le r<1,
    \end{split}
    \end{equation*}
for a suitably fixed $K=K(\om)>1$, and thus (ii) is satisfied. Conversely, (ii) implies
    \begin{equation*}
    \begin{split}
    C\widehat{\nu}(r)
    &\ge\int_r^1\frac{\om(t)\widehat{\nu}(t)}{\widehat{\om}(t)}\,dt
    \ge\int_r^{\frac{1+r}{2}}\frac{\om(t)\widehat{\nu}(t)}{\widehat{\om}(t)}\,dt
    \ge\widehat{\nu}\left(\frac{1+r}{2}\right)\log\frac{\widehat{\om}(r)}{\widehat{\om}\left(\frac{1+r}{2}\right)},\quad 0\le r<1,
    \end{split}
    \end{equation*}
and since $\nu\in\DDD\subset\DD$ by the hypothesis, we deduce $\widehat{\om}(r)\lesssim\widehat{\om}\left(\frac{1+r}{2}\right)$ for all $0\le r<1$. Thus $\om\in\DD$.

Let $\om\in\DD$ and $0\le r<1$, and consider $\r_n=\r_n(\om,K)$ for all $n\in\N\cup\{0\}$. Fix $k=k(\om,K)\in\N\cup\{0\}$ such that $\rho_k\le r<\rho_{k+1}$. Then
    $$
    \int_0^r\frac{\omega(t)}{\widehat{\omega}(t)\widehat{\nu}(t)}\,dt
    =\sum_{n=0}^{k-1}\int_{\rho_n}^{\rho_{n+1}}\frac{\omega(t)}{\widehat{\omega}(t)\widehat{\nu}(t)}\,dt
    +\int_{\rho_k}^{r}\frac{\omega(t)}{\widehat{\omega}(t)\widehat{\nu}(t)}\,dt,\quad 0\le r<1,
    $$
where, by Lemma~\ref{Lemma:weights-in-D}, applied to $\nu\in\DDD\subset\Dd$, and Lemma~\ref{Lemma:weights-in-D-hat}(v), applied to $\om$,
    \begin{equation*}
    \begin{split}
    \sum_{n=0}^{k-1}\int_{\rho_n}^{\rho_{n+1}}\frac{\omega(t)}{\widehat{\omega}(t)\widehat{\nu}(t)}\,dt
    &\le\sum_{n=0}^{k-1}\frac{1}{\widehat{\nu}(\rho_{n+1})}\int_{\rho_n}^{\rho_{n+1}}\frac{\omega(t)}{\widehat{\omega}(t)}\,dt\\
    &\lesssim\sum_{n=0}^{k-1}\frac{(1-\r_k)^\a}{\widehat{\nu}(\r_k)}\frac{1}{(1-\rho_{n+1})^\a} \log\left(\frac{\widehat{\omega}(\rho_n)}{\widehat{\omega}(\rho_{n+1})}\right)\\
    &\le\log K\frac{(1-\r_k)^\a}{\widehat{\nu}(r)}\sum_{n=0}^{k-1} \frac{1}{(C^\a)^{k-1-n}(1-\rho_{k})^\a}\\
    &\le\frac{\log K}{\widehat{\nu}(r)}\sum_{n=0}^{\infty}\frac1{(C^\a)^{n}}
    =\frac{\log K}{\widehat{\nu}(r)}\frac{C^\a}{C^\a-1},\quad k\in\N,
    \end{split}
    \end{equation*}
for some $\alpha=\alpha(\nu)>0$ and for a suitably fixed $K=K(\om)>1$, and similarly,
    $$
    \int_{\rho_k}^{r}\frac{\omega(t)}{\widehat{\omega}(t)\widehat{\nu}(t)}\,dt
    \le\frac{1}{\widehat{\nu}(r)}\log\left(\frac{\widehat{\omega}(\rho_k)}{\widehat{\omega}(r)}\right)
    \le\frac{\log K}{\widehat{\nu}(r)},\quad k\in\N\cup\{0\}.
    $$
The statement (iii) follows from these estimates.

Conversely, by replacing $r$ by $\frac{1+r}{2}$ in (iii) we obtain
    \begin{equation*}
    \begin{split}
    \frac{C}{\widehat{\nu}\left(\frac{1+r}{2}\right)}
    &\ge\int_0^{(1+r)/2}\frac{\omega(t)}{\widehat{\omega}(t)\widehat{\nu}(t)}\,dt
    \ge\int_r^{(1+r)/2}\frac{\omega(t)}{\widehat{\omega}(t)\widehat{\nu}(t)}\,dt
    \ge\frac{1}{\widehat{\nu}(r)}\log\frac{\widehat{\omega}(r)}{\widehat{\omega}\left(\frac{1+r}{2}\right)},\quad 0\le r<1,
    \end{split}
    \end{equation*}
and since $\nu\in\DDD\subset\DD$ by the hypothesis, we deduce $\widehat{\om}(r)\lesssim\widehat{\om}\left(\frac{1+r}{2}\right)$ for all $0\le r<1$. Thus $\om\in\DD$.
\end{proof}

\begin{lemma}\label{le:nuhiyperbolic}
Let $\om,\nu\in\DDD$, and denote $\sigma=\sigma_{\om,\nu}=\om\widehat{\nu}/\widehat{\om}$. Then $\widehat{\sigma}\asymp\widehat{\nu}$ on $[0,1)$, and hence $\sigma\in\DDD$.
\end{lemma}

\begin{proof}
Lemma~\ref{lemma:d-hat-new}(ii) implies $\widehat{\sigma}\lesssim\widehat{\nu}$ on $[0,1)$. The argument used to prove (i)$\Rightarrow$(ii) in the said lemma shows that $\widehat{\sigma}\gtrsim\widehat{\nu}$ on $[0,1)$, provided $\om\in\Dd$ and $\nu\in\DDD$. Thus $\widehat{\sigma}\asymp\widehat{\nu}$, and $\sigma\in\DDD$ by Lemmas~\ref{Lemma:weights-in-D-hat}(ii) and~\ref{Lemma:weights-in-D}.
\end{proof}

The next lemma says that in many instances concerning $A^p$-norms we may replace $\om$ by $\widetilde\om=\widehat{\om}/(1-|\cdot|)$ if $\om\in\DDD$. This result has the flavor of radial Carleson measures and indeed can be established by appealing to the characterization of Carleson measures for the Bergman space $A^p_\om$ induced by $\om\in\DD$ given in \cite{PR2015/2}. That approach requires showing that the involved weights belong to $\DD$, which is of course the case, and thus involves more calculations than the simple proof given below.

\begin{lemma}\label{LemmaDeMierda}
Let $0<p<\infty$, $\om\in\DDD$ and $-\alpha<\kappa<\infty$, where $\alpha=\alpha(\om)>0$ is that of Lemma~\ref{Lemma:weights-in-D}. Then
    \begin{equation}\label{11}
    \int_\D|f(z)|^p(1-|z|)^{\kappa}\om(z)\,dA(z)\asymp\int_\D|f(z)|^p(1-|z|)^{\kappa-1}\widehat{\om}(z)\,dA(z),\quad f\in\H(\D).
    \end{equation}
\end{lemma}

\begin{proof}
The function $(1-|\cdot|)^{\kappa-1}\widehat{\om}$ is a weight for each $\kappa>-\alpha$ by Lemma~\ref{Lemma:weights-in-D}. Therefore an integration by parts shows that \eqref{11} is equivalent to
    $$
    \int_0^1\frac{\partial}{\partial r}M_p^p(r,f)\left(\int_r^1(1-t)^{\kappa}\om(t)\,dt\right)\,dr\asymp
    \int_0^1\frac{\partial}{\partial r}M_p^p(r,f)\left(\int_r^1(1-t)^{\kappa-1}\widehat{\om}(t)\,dt\right)\,dr.
    $$
Another integration by parts reveals that both integrals from $r$ to 1 above are bounded by a constant times $\widehat{\om}(r)(1-r)^\kappa$. But Lemma~\ref{Lemma:weights-in-D-hat}(ii) implies
    $$
    \int_r^1(1-t)^{\kappa-1}\widehat{\om}(t)\,dt
    \gtrsim\frac{\widehat{\om}(r)}{(1-r)^{\b(\om)}}\int_r^1(1-t)^{\kappa-1+\b(\om)}\,dt
    \asymp\widehat{\om}(r)(1-r)^\kappa,\quad 0\le r<1,
    $$
and
    \begin{equation*}
    \begin{split}
    \int_r^1(1-t)^{\kappa}\om(t)\,dt\gtrsim
    \frac{\widehat{\om}(r)}{(1-r)^{\b(\om)}}\int_r^1\frac{\om(t)(1-t)^{\kappa+\b(\om)}}{\widehat{\om}(t)}\,dt
    \asymp\widehat{\om}(r)(1-r)^\kappa,\quad 0\le r<1,
    \end{split}
    \end{equation*}
by Lemma~\ref{le:nuhiyperbolic}. The assertion follows.
\end{proof}

The last auxiliary results shows that each radial weight in the Bekoll\'e-Bonami class $B_q$ belongs to $\DDD$, and for each $\nu\in\DDD$ the maximal Bergman projection
    $$
    P^+_\nu(f)(z)=\int_\D f(\z)|B^\nu_z(\z)|\nu(\z)\,dA(\z),\quad z\in\D,
    $$
is bounded on $L^q_\nu$. It is worth noticing that obviously $\DDD\not\subset\cup_{1<q<\infty}B_q$ because $\nu\in\DDD$ may vanish on a set of positive measure.

\begin{proposition}\label{pr:bqDDD}
Let $1<q<\infty$ and $\nu\in B_q$ a radial weight. Then $\nu\in\DDD$. Moreover, $P^+_\nu:L^q_\nu\to L^q_\nu$ is bounded for all $\nu\in\DDD$.
\end{proposition}

\begin{proof}
If $\nu\in B_q$, then by \cite{BB} there exists $\beta>-1$ such that
    \begin{equation*}
    \begin{split}
    &\left(\int_{S(a)}\nu(z)\,dA(z)\right)^\frac1q
    \left(\int_{S(a)}\left(\frac{(1-|z|)^\b}{\nu(z)}\right)^{\frac{q'}{q}}(1-|z|)^\b\,dA(z)\right)^{\frac{1}{q'}}\lesssim(1-|a|)^{(2+\b)},\quad a\in\D.
    \end{split}
    \end{equation*}
Since $\nu$ is radial, this condition easily implies $\nu\in\DDD$.

Let now $1<q<\infty$ and $\nu\in\DDD$, and define $h=\widehat{\nu}^{-\frac{1}{qq'}}$. Then $\int_t^1h(s)^{q'}\nu(s)\,ds\asymp \widehat{\nu}(t)^{\frac{1}{q'}}$ for all $0\le t<1$. Therefore Lemma~\ref{Lemma:weights-in-D} yields
    \begin{equation}\label{jap:1111}
    \int_0^r\frac{\int_t^1h(s)^{q'}\nu(s)\,ds}{\widehat{\nu}(t)(1-t)}\,dt
    \asymp\int_0^r\frac{dt}{\widehat{\nu}(t)^\frac1q(1-t)}
    \lesssim\frac{1}{\widehat{\nu}(r)^\frac1q}=h^{q'}(r),\quad 0\le r <1.
    \end{equation}
Moreover, by symmetry, \eqref{jap:1111} with $q'$ in place of $q$ is satisfied. Since $\nu\in\DD$, we may apply \cite[Theorem~1]{PR2016/1} and \eqref{jap:1111} to deduce
    $$
    \int_{\D}|B^\nu_z(\z)|h^{p'}(\z)\nu(\z)\,dA(\z)\lesssim h^{p'}(z),\quad z\in\D,
    $$
and
    $$
    \int_{\D}|B^\nu_z(\z)|h^{p}(z)\nu(z)\,dA(z)\lesssim h^{p}(\z),\quad \z\in\D.
    $$
It follows from  Schur's test~\cite[Theorem~3.6]{Zhu} that the maximal Bergman projection $P^+_\nu:L^p_\nu\to L^p_\nu$ is bounded.
\end{proof}

\section{Some spaces of functions}
\label{Some spaces of functions}

Recall that
    \begin{equation}\label{eq:gamma}
    \g(z)=\g_{\om,\nu,p,q}(z)=\frac{\widehat{\nu}(z)^\frac1q(1-|z|)^{\frac1q}}{\widehat{\om}(z)^\frac1p(1-|z|)^{\frac1p}},\quad z\in\D,
    \end{equation}
and $\fgrnu(z)=\frac{\int_{\Delta(z,r)}f(\z)\nu(\z)\,dA(\z)}{\nu(\Delta(z,r))}$ for $f\in L^1_{\nu,{\rm loc}}$, and
    $$
    \MO_{\nu,q,r}(f)(z)=\left(\frac{1}{\nu(\Delta(z,r))} \int_{\Delta(z,r)}|f(\z)-\fgrnu(z)|^q\nu(\z)\,dA(\z)\right)^{\frac{1}{q}}
    $$
for all $z\in\D$. If $\nu\in\Dd$, then by the definition there exists $K=K(\nu)>1$ and $C=C(\nu)>1$ such that
    $$
    \int_r^{1-\frac{1-r}{K}}\nu(s)\,ds\ge(C-1)\widehat{\nu}\left(1-\frac{1-r}{K}\right)>0,\quad 0\le r<1.
    $$
It follows that there exists $r_\nu\in(0,\infty)$ such that $\nu(\Delta(z,r))>0$ for all $z\in\D$ if $r\ge r_\nu$.

The space $\BMO(\Delta)=\BMO(\Delta)_{\om,\nu,p,q,r}$ consists of $f\in L^q_{\nu,{\rm loc}}$ such that
    $$
    \|f\|_{\BMO(\Delta)}=\sup_{z\in\D}\left(\MO_{\nu,q,r}(f)(z)\g(z)\right)<\infty.
    $$
The following lemma is easy to establish; see \cite[Lemma~3.1]{PZZ} for a similar result.

\begin{lemma}\label{le:BMOr}
Let $1\le p,q<\infty$, $\om$ a radial weight, $\nu\in\Dd$ and $r_\nu\le r<\infty$.
Then
    $$
    \MO_{\nu,q,r}(f)(z)\le2\left(\frac{1}{\nu(\Delta(z,r))}\int_{\Delta(z,r)}|f(\z)-\lambda|^q\nu(\z)\,dA(\z)\right)^\frac1q,\quad z\in\D,\quad \lambda\in\C,\quad f\in L^q_\nu,
    $$
and therefore $f\in L^q_\nu$ belongs to $\BMO(\Delta)$ if and only if for each $z\in\D$ there exists $\lambda_z\in\C$ such that
    $$
    \sup_{z\in\D}\left(\frac{\gamma(z)^q}{\nu(\Delta(z,r))}\int_{\Delta(z,r)}|f(\z)-\lambda_z|^q\nu(\z)\,dA(\z)\right)<\infty.
    $$
\end{lemma}

For $0<p,q<\infty$, $0\le \tau<\infty$ and radial weights $\om,\nu$, let
    \begin{equation}\label{gamma}
    \Gamma_\tau(z,\z)=\frac{\left(\frac{|1-\overline{z}\z|^2}{\max\{1-|z|^2,1-|\z|^2\}}\right)^{\frac{1}{p}-\frac{\tau+1}{q}}
    \widehat{\om}\left(1-\frac{2|1-\overline{z}\z|^2}{\max\{1-|z|^2,1-|\z|^2\}}\right)^{\frac{1}{p}}}
    {\min\left\{ \frac{\widehat{\nu}(z)}{(1-|z|)^\tau},
    \frac{\widehat{\nu}(\z)}{(1-|\z|)^\tau}\right\}^{\frac1q}},\quad z,\z\in\D,
    \end{equation}
with the understanding that $\widehat{\om}(t)=\widehat{\om}(0)$ when $t<0$. The following lemma explains the behavior of $\Gamma_\tau$ near the diagonal.

\begin{lemma}\label{diag}
Let $0<p,q,r<\infty$, $0\le \tau<\infty$ and $\om,\nu\in\DD$.
Then
    $$
    \Gamma_\tau(z,\z)\asymp \gamma(z)^{-1}\asymp \gamma(\z)^{-1},\quad \beta(z,\z)\le r.
    $$
\end{lemma}

\begin{proof}
Clearly
    $$
    |1-\overline{z}\z|\asymp1-|z|\asymp1-|\z|,\quad \b(z,\z)\le r,
    $$
and hence there exist $0<m_r<1<M_r<\infty$ such that
    $$
    m_r(1-|z|)\le\frac{2|1-\overline{z}\z|^2}{\max\{1-|z|^2,1-|\z|^2\}}\le M_r(1-|z|),\quad \b(z,\z)\le r.
    $$
Since $\om\in\DD$ by the hypothesis, and $\widehat{\om}(t)=\widehat{\om}(0)$ for $t<0$, Lemma~\ref{Lemma:weights-in-D-hat}(ii) implies
    $$
    \widehat{\om}(z)
    \le\frac{C}{m_r^\b}\widehat{\om}(1-m_r(1-|z|))
    \le\frac{C}{m_r^\b}\widehat{\om}\left(1-\frac{2|1-\overline{z}\z|^2}{\max\{1-|z|^2,1-|\z|^2\}}\right),\quad \b(z,\z)\le r,
    $$
and
    $$
    \widehat{\om}\left(1-\frac{2|1-\overline{z}\z|^2}{\max\{1-|z|^2,1-|\z|^2\}}\right)
    \le CM_r^\b\widehat{\om}(1-M_r(1-|z|))
    \le CM_r^\b\widehat{\om}(z),\quad \b(z,\z)\le r,
    $$
for some $C=C(\om)>0$ and $\b=\b(\om)>0$. Further, $\widehat{\nu}(z)\asymp\widehat{\nu}(\z)$ and $\widehat{\om}(z)\asymp\widehat{\om}(\z)$ if $\b(z,\z)\le r$ by Lemma~\ref{Lemma:weights-in-D-hat}(ii).
The assertion follows from these estimates.
\end{proof}

For continuous $f:\D\to\C$ and $0<r<\infty$, define
    $$
    \Omega_r f(z)=\sup\{|f(z)-f(\z)|:\beta(z,\z)<r\},\quad z\in\D,
    $$
and let $\BO(\Delta)=\BO(\Delta)_{\om,\nu,p,q,r}$ denote the space of those $f$ such that
    $$
    \|f\|_{\BO(\Delta)}=\sup_{z\in\D}\left(\Omega_r f(z)\gamma(z)\right)<\infty.
    $$
Lemma~\ref{bo} shows that the space $\BO(\Delta)=\BO(\Delta)_{\om,\nu,p,q,r}$ is independent of $r$.

\begin{lemma}\label{bo}
Let $0<p\le q<\infty$, $0<r<\infty$, $\om,\nu\in\Dd$ and $\g(z)=\g_{\om,\nu,p,q}(z)=\frac{\widehat{\nu}(z)^\frac1q(1-|z|)^{\frac1q}}{\widehat{\om}(z)^\frac1p(1-|z|)^{\frac1p}}$.
Let $f:\D\to\C$ be continuous, and $0<\tau<\min\{q\a(\om)/p, \a(\nu)\}$, where $\alpha(\nu)$ and $\a(\om)$ are those from Lemma~\ref{Lemma:weights-in-D}. Then the following statements are equivalent:
    \begin{itemize}
    \item[\rm (i)] $f\in\BO(\Delta)$;
    \item[\rm (ii)] $|f(z)-f(\z)|\lesssim\|f\|_{\BO(\Delta)}(1+\beta(z,\z))\Gamma_\tau(z,\z)$ for all $z,\z\in\D$.
    \end{itemize}
\end{lemma}

\begin{proof}
Lemma~\ref{diag} shows that (ii) implies (i). For the converse, assume (i), that is,
    \begin{equation}\label{6}
    |f(z)-f(\z)|\gamma(z)\le\|f\|_{\BO(\Delta)},\quad \b(z,\z)<r.
    \end{equation}
The estimate (ii) for $\beta(z,\z)\le r$ then follows from Lemma~\ref{diag}. If $\beta(z,\z)>r$, let $N=\max\{n\in\N:n\le\b(z,\z)/r+1\}$, and pick up $N+1$ points from the geodesic joining $z$ and $\z$ such that $\b(z_j,z_{j+1})=\b(z,\z)/N<r$ for all $j=0,\ldots,N-1$. Then, as the hyperbolic distance is additive along geodesics, \eqref{6} yields
    $$
    |f(z)-f(\z)|
    \le\sum_{j=0}^{N-1}|f(z_{j})-f(z_{j+1})|
    \le\|f\|_{\BO(\Delta)}\sum_{j=0}^{N-1}\frac{\widehat{\om}(z_j)}{\widehat{\nu}(z_j)}(1-|z_j|)^{\frac1p-\frac1q}.
    $$
Next, observe that
    \begin{equation}\label{eq:h}
    1-|z_j|\le\frac{2|1-\overline{z}\z|^2}{\max\{1-|z|^2,1-|\z|^2\}},\quad j=0,\ldots, N;
    \end{equation}
see the proof of \cite[Lemma~3.2]{PZZ} for details. This together with the inequality $\frac{1}{p}-\frac{1}{q}\ge0$ gives
    \begin{equation*}
    \begin{split}
    |f(z)-f(\z)|
    &\le\|f\|_{\BO(\Delta)}
    \left(\frac{2|1-\overline{z}\z|^2}{\max\{1-|z|^2,1-|\z|^2\}}\right)^{\frac{1}{p}-\frac{1}{q}}
    \sum_{j=0}^{N-1}\frac{\widehat{\om}(z_j)^{\frac{1}{p}}}{\widehat{\nu}(z_j)^{\frac1q}}\\
    &=\|f\|_{\BO(\Delta)}
    \left(\frac{2|1-\overline{z}\z|^2}{\max\{1-|z|^2,1-|\z|^2\}}\right)^{\frac{1}{p}-\frac{1}{q}}
    \sum_{j=0}^{N-1} \frac{\widehat{\om}(z_j)^{\frac{1}{p}}}{(1-|z_j|)^{\frac{\tau}{q}}}\frac{(1-|z_j|)^{\frac{\tau}{q}}}{\widehat{\nu}(z_j)^{\frac1q}}.
    \end{split}
    \end{equation*}
The election of $\tau$ together with Lemma~\ref{Lemma:weights-in-D} shows that the functions $\widehat{\om}(r)/(1-r)^{\frac{p\tau}{q}}$ and $\widehat{\nu}(r)/(1-r)^\tau$ are essentially decreasing on $[0,1)$. Therefore the inequalities \eqref{eq:h} and $|z_j|\le\max\{|z|,|\z|\}$ yield
    \begin{equation*}
    \begin{split}
    |f(z)-f(\z)|
    &\lesssim \|f\|_{\BO(\Delta)}\left(\frac{2|1-\overline{z}\z|^2}{\max\{1-|z|^2,1-|\z|^2\}}\right)^{\frac{1}{p}-\frac{\tau+1}{q}}\\
    &\quad\cdot\widehat{\om}\left(1-\frac{2|1-\overline{z}\z|^2}{\max\{1-|z|^2,1-|\z|^2\}}\right)^{\frac{1}{p}}
    \sum_{j=0}^{N-1}\frac{(1-|z_j|)^{\frac{\tau}{q}}}{\widehat{\nu}(z_j)^{\frac{1}{q}}}\\
    &\lesssim\|f\|_{\BO(\Delta)} \Gamma_\tau(z,\z)N
    \lesssim\|f\|_{\BO(\Delta)}(1+\beta(z,\z))\Gamma_\tau(z,\z),\quad \beta(z,\z)>r.
    \end{split}
    \end{equation*}
Therefore (ii) is satisfied.
\end{proof}

For $0<p,q<\infty$, $0<r<\infty$ and radial weights $\om,\nu$, the space $\BA(\Delta)=\BA(\Delta)_{\om,\nu,p,q,r}$ consists of $f\in L^q_{\nu,{\rm loc}}$ such that
    $$
    \|f\|_{\BA(\Delta)}
    =\sup_{z\in\D}\left(\left(\frac{1}{\nu(\Delta(z,r))}\int_{\Delta(z,r)}|f(\z)|^q\nu(\z)\,dA(\z)\right)^{\frac1q}\gamma(z)\right)<\infty.
    $$
For $c,\sigma\in\mathbb{R}$ and a radial weight $\nu$, the general Berezin transform of $\varphi\in L^1_{\nu(1-|\cdot|)^\s}$ is defined by
    $$
    B(\vp)(z)=B_{\nu,c,\sigma}(\varphi)(z)
    =\frac{(1-|z|^2)^{c+1}}{\widehat{\nu}(z)} \int_\D \varphi(\z) \frac{(1-|\z|^2)^\sigma}{|1-z\overline{\z}|^{2+c+\sigma}}\nu(\z)\,dA(\z),\quad z\in\D.
    $$
The next lemma shows, in particular, that the space $\BA(\Delta)=\BA(\Delta)_{\om,\nu,p,q,r}$ is independent of $r$ as long as $r$ is sufficiently large depending on $\nu\in\DDD$.

\begin{lemma}\label{ba}
Let $0<p\le q<\infty$, $0<r<\infty$ and $\om,\nu\in\DDD$, $\g(z)=\g_{\om,\nu,p,q}(z)=\frac{\widehat{\nu}(z)^\frac1q(1-|z|)^{\frac1q}}{\widehat{\om}(z)^\frac1p(1-|z|)^{\frac1p}}$. If $f\in L^q_\nu$, then the following statements are equivalent:
    \begin{itemize}
    \item[\rm (i)] There exists $r_0=r_0(\nu)>0$ such that $f\in\BA(\Delta)=\BA(\Delta)_{\om,\nu,p,q,r}$ for all $r\ge r_0$;
    \item[\rm (ii)] $|f|^q\nu dA$ is a $q$-Carleson measure for $A^p_\om$;
    \item[\rm (iii)] The identity operator $Id:A^p_\om\to L^q_{|f|^q\nu}$ is bounded;
    \item[\rm (iv)] The multiplication operator $M_f(g)=fg$ is bounded from $A^p_\om$ to $L^q_\nu$;
    \item[\rm (v)] $\sup_{z\in\D}\gamma(z)^{q}B(|f|^q)(z)<\infty$ for all $\sigma>1-\frac{q}{p}(1+\a)$ and $c>\max\{-1-\sigma, \frac{q}{p}(1+\b)-2\}$, where $\alpha=\alpha(\om)>0$ and $\beta=\beta(\om)>0$ are those of Lemmas~\ref{Lemma:weights-in-D-hat}(ii) and~\ref{Lemma:weights-in-D}.
    \end{itemize}
\end{lemma}

\begin{proof}
It is obvious that (ii), (iii) and (iv) are equivalent by the definitions. Assume (ii) is satisfied, that is,
    \begin{equation}\label{7}
    \left(\int_\D|g(\z)|^q|f(\z)|^q\nu(\z)\,dA(\z)\right)^\frac1q\lesssim\|g\|_{A^p_\om},\quad g\in A^p_\om.
    \end{equation}
For $z\in\D$, let $g_z(\z)=\left(\frac{1-|z|}{1-\overline{z}\z}\right)^{\frac{\lambda+1}{p}}$, where $\lambda=\lambda(\om)>0$ is that of Lemma~\ref{Lemma:weights-in-D-hat}(iv). Further, since $\nu\in\Dd$ by the hypothesis, there exists $r_\nu\in(0,\infty)$ such that $\nu(\Delta(z,r))>0$ for all $r\ge r_\nu$. For $g=g_z$ and $r\ge r_\nu$, \eqref{7} yields
    $$
    \left(\frac{1}{\nu(\Delta(z,r))}\int_{\Delta(z,r)}|f(\z)|^q\nu(\z)\,dA(\z)\right)^\frac1q
    \lesssim\frac{\|g_z\|_{A^p_\om}}{\nu(\Delta(z,r))^\frac1q}\lesssim\frac{\left(\widehat{\om}(z)(1-|z|)\right)^\frac1p}{\nu(\Delta(z,r))^\frac1q},\quad z\in\D.
    $$
But since $\nu\in\DDD$, applications of Lemmas~\ref{Lemma:weights-in-D-hat}(ii) and~\ref{Lemma:weights-in-D} show that
    \begin{equation}\label{8}
    \nu(\Delta(z,r))\asymp\widehat{\nu}(z)(1-|z|),\quad z\in\D,
    \end{equation}
if $r$ is sufficiently large. It follows that $f\in\BA(\Delta)=\BA(\Delta)_{\om,\nu,p,q,r}$ for all such $r$, and thus (i) is satisfied.

Conversely, if (i) is satisfied, then by using \eqref{8} we deduce
    $$
    \left(\int_{\Delta(z,r)}|f(\z)|^q\nu(\z)\,dA(\z)\right)^\frac1q\lesssim\widehat{\om}(z)^\frac1p(1-|z|)^\frac1p,\quad z\in\D.
    $$
Therefore $|f|^q\nu dA$ is a $q$-Carleson measure for $A^p_\om$ by \cite[Theorem~3]{PRS18}.

By integrating only over $\Delta(z,r)$ in (v) and using \eqref{8} we obtain (i) from (v). To complete the proof of the lemma, it remains to show the converse implication. To do this, pick up a sequence $\{a_j\}$ and $0<r<\infty$ in accordance with \cite[Lemma~4.7]{Zhu}, and observe that $\widehat{\om}$ is essentially constant in each hyperbolically bounded region by Lemma~\ref{Lemma:weights-in-D-hat}(ii). Then  by using \eqref{8}, the hypothesis (i), the election of $c$ and $\sigma$, and
finally Lemmas~\ref{Lemma:weights-in-D-hat}(ii) and \ref{Lemma:weights-in-D}, we deduce
    \begin{equation*}
    \begin{split}
    \frac{\widehat{\nu}(z)B(|f|^q)(z)}{(1-|z|^2)^{c+1}}
    &\lesssim\sum_{j=1}^\infty\int_{\Delta(a_j,r)}|f(\z)|^q\frac{(1-|\z|^2)^\sigma}{|1-z\overline{\z}|^{2+c+\sigma}}\nu(\z)\,dA(\z)\\
    &\lesssim\sum_{j=1}^\infty \frac{(1-|a_j|^2)^\sigma}{|1-z\overline{a_j}|^{2+c+\sigma}}\int_{\Delta(a_j,r)}|f(\z)|^q\nu(\z)\,dA(\z)\\
    &\lesssim\sum_{j=1}^\infty\frac{(1-|a_j|^2)^{\sigma+1}\widehat{\nu}(a_j)}{|1-z\overline{a_j}|^{2+c+\sigma}\nu(\Delta(a_j,r))}
    \int_{\Delta(a_j,r)}|f(\z)|^q\nu(\z) dA(\z)\\
    &\lesssim \sum_{i=1}^\infty \frac{(1-|a_i|^2)^{\sigma+1}\widehat{\nu}(a_i)}{|1-z\overline{a_i}|^{2+c+\sigma}\gamma(a_i)^q}
    \asymp\sum_{j=1}^\infty \frac{(1-|a_j|^2)^{\sigma+\frac{q}{p}}\widehat{\om}(a_j)^{\frac{q}{p}}}{|1-z\overline{a_j}|^{2+c+\sigma}}\\
    &\lesssim\int_\D\frac{(1-|u|^2)^{\sigma+\frac{q}{p}-2}\widehat{\om}(u)^{\frac{q}{p}}}{|1-z\overline{u}|^{2+c+\sigma}}\,dA(u)\\
    &\lesssim\int_{0}^{|z|}\frac{\widehat{\om}(t)^{\frac{q}{p}}}{(1-t)^{c+3-\frac{q}{p}}}\,dt
    +\frac{1}{(1-|z|)^{c+\sigma+1}}\int_{|z|}^{1}(1-t)^{\sigma+\frac{q}{p}-2}\widehat{\om}(t)^{\frac{q}{p}}\,dt\\
    &\lesssim\frac{\widehat{\om}(|z|)^{\frac{q}{p}}}{(1-|z|)^{c+2-\frac{q}{p}}}
    \asymp\frac{\widehat{\nu}(z)}{(1-|z|^2)^{c+1}\gamma(z)^{q}},\quad z\in\D,
    \end{split}
    \end{equation*}
and thus (v) is satisfied.
\end{proof}

With these preparations we are ready to show that $\BMO(\Delta)=\BA(\Delta)+\BO(\Delta)$. This follows from the case (ii) of the next theorem.

\begin{theorem}\label{th:BMOdecom}
Let $1\le p\le q<\infty$, $\om,\nu\in\DDD$, $\g(z)=\g_{\om,\nu,p,q}(z)=\frac{\widehat{\nu}(z)^\frac1q(1-|z|)^{\frac1q}}{\widehat{\om}(z)^\frac1p(1-|z|)^{\frac1p}}$ and $f\in L^q_\nu$. Further, let $r\ge r_\nu$, $\sigma>0$ and
    \begin{equation*}
    \begin{split}
    c>2\frac{q}{p}\left(\b(\om)+1\right)+\sigma+\max\left\{2\b(\nu),\gamma(\nu)\right\},
    \end{split}
    \end{equation*}
where $\b(\om),\b(\nu),\g(\nu)>0$ are associated to $\nu$ and $\om$ via Lemma~\ref{Lemma:weights-in-D-hat}(ii)(iii). Then the following statements are equivalent:
    \begin{itemize}
    \item[\rm (i)] There exists $r_0=r_0(\nu)\ge r_\nu$ such that $f\in\BMO(\Delta)=\BMO(\Delta)_{\om,\nu,p,q,r}$ for all $r\ge r_0$;
    \item[\rm (ii)] $f=f_1+f_2$, where $f_1\in\BA(\Delta)$ and $f_2=\fgrnu \in\BO(\Delta)$;
    \item[\rm (iii)] $\displaystyle \sup_{z\in\D}\left(B(|f-\widehat{f}_{r,\nu}(z)|^q)\gamma(z)^q\right)<\infty$;
    \item[\rm (iv)] For each $z\in\D$ there exists $\lambda_z\in\C$ such that $\displaystyle \sup_{z\in\D}\left(B(|f-\lambda_z|^q)\gamma(z)^q\right)<\infty$.
    \end{itemize}
    \end{theorem}
\begin{proof}
Obviously, (iii) implies (iv). Next assume (iv). The relation \eqref{8} shows that there exists $r_0=r_0(\nu)>0$ such that
    \begin{equation*}
    \begin{split}
    &\frac{1}{\nu\left(\Delta(z,r)\right)}\int_{\Delta(z,r)}|f(\z)-\lambda_z|^q\nu(\z)\,dA(\z)\\
    &\lesssim\frac{(1-|z|)^{c+1}}{\widehat{\nu}(z)}\int_\D|f(\z)-\lambda_z|^q\frac{(1-|\z|^2)^\sigma}{|1-z\overline{\z}|^{2+c+\sigma}}\nu(\z)\,dA(\z),
    \quad z\in\D,\quad r_0\le r<\infty,
    \end{split}
    \end{equation*}
which together with Lemma~\ref{le:BMOr} shows that (i) is satisfied.

Assume now (i), and let $f_2=\fgrnu$. Since $f\in L^q_\nu$, $q\ge 1$ and $r\ge r_\nu$, the function $f_2$ is well defined and continuous.
Since $\om,\nu\in\DDD$ by the hypothesis, one may use Lemmas~\ref{Lemma:weights-in-D-hat}(ii) and \ref{Lemma:weights-in-D} together with the argument in \cite[1651-1652]{PZZ} with minor modifications to show that $f_2=\fgrnu\in\BO(\Delta)$ and $f_1=f-\fgrnu\in\BA(\Delta)$. Thus (ii) is satisfied.

To complete the proof it suffices to show that (ii) implies (iii), so assume $f=f_1+f_2$, where $f_1\in\BA(\Delta)$ and $f_2=\fgrnu \in\BO(\Delta)$. Since $\widehat{f}_{r,\nu}=\widehat{f_1}_{r,\nu}+\widehat{f_2}_{r,\nu}$, it suffices to prove the condition in (iii) for $f_1$ and $f_2$ separately. First observe that by Lemma~\ref{Lemma:weights-in-D-hat}(iii) the constant function $1$ satisfies
    \begin{equation*}
    \begin{split}
    B(1)(z)\lesssim\frac{(1-|z|)^{c+1}}{\widehat{\nu}(z)}
    \left(\int_0^{|z|}\frac{\nu(t)}{(1-t)^{1+c}}\,dt + \frac{1}{(1-|z|)^{1+c+\sigma}}\int_{|z|}^1 (1-t)^\sigma \nu(t)\,dt\right)\lesssim 1,\quad z\in\D,
    \end{split}
    \end{equation*}
because $c>\max\{\gamma(\nu),\s\}-1$ by the hypothesis. This together with H\"older's inequality and Lemma~\ref{ba} yields
    \begin{equation*}
    \begin{split}
    B\left(\left|f_1-\widehat{f_1}_{r,\nu}(z)\right|^q\right)\gamma(z)^q
    &\lesssim\left(B(|f_1|^q)(z)+|\widehat{f_1}_{r,\nu}(z)|^q\right)\gamma(z)^q\\
    &\le\left(B(|f_1|^q)(z)+\widehat{|{f_1}|^q}_{r,\nu}(z)\right)\gamma(z)^q\lesssim1,\quad z\in\D,
    \end{split}
    \end{equation*}
and thus (iii) for $f_1\in\BA(\Delta)$ is satisfied.

To deal with $f_2\in\BO(\Delta)$, pick up $\tau$ satisfying the hypothesis of Lemma~\ref{bo}. Then
    \begin{equation*}
    \begin{split}
    |f_2(\z)-\widehat{f_2}_{r,\nu}(z)|
    &=\left|\frac{1}{\nu(\Delta(z,r)}\int_{\Delta(z,r)}(f_2(\z)-f_2(u))\nu(u)\,dA(u) \right|\\
    &\le\frac{1}{\nu(\Delta(z,r)}\int_{\Delta(z,r)} |f_2(\z)-f_2(u)|\nu(u)\,dA(u)\\
    &\lesssim\frac{1}{\nu(\Delta(z,r)}\int_{\Delta(z,r)}(1+\beta(\z,u))\Gamma_\tau(\z,u)\nu(u)\,dA(u)\\
    &\lesssim(1+\beta(z,\z))\Gamma_\tau(z,\z),\quad z,\z\in\D,
    \end{split}
    \end{equation*}
because $\Gamma_\tau(\z,u)\asymp\Gamma_\tau(z,\z)$ for all $u\in\Delta(z,r)$ by Lemma~\ref{Lemma:weights-in-D-hat}(ii); see the proof of Lemma~\ref{diag} for similar estimates. Hence it suffices to show that
    \begin{equation}\label{eq:f1}
    \frac{(1-|z|)^{c+1}\gamma(z)^q}{\widehat{\nu}(z)}
    \int_\D|(1+\beta(z,\z))\Gamma_\tau(z,\z)|^q\frac{(1-|\z|^2)^\sigma}{|1-z\overline{\z}|^{2+c+\sigma}}\nu(\z)\,dA(\z)\lesssim 1,\quad z\in\D,
    \end{equation}
to obtain (iii) for $f_2\in\BO(\Delta)$. The proof of \eqref{eq:f1} is involved and will be divided into four separate cases. Before dealing with each case, we observe that since $\beta(z,\z)$ grows logarithmically, we may pick up $0<\delta<\min\left\{\sigma,\frac{q}{p}\b(\om)+\b(\nu)+\frac{\s}{2}\right\}$ and a constant $C=C(\delta)>0$ such that
    \begin{equation}\label{eq:delta}
    1+\beta(z,\z)
    \le C\left|(1-|\varphi_z(\z)|\right)^{-\frac{\delta}{q}}
    =C\left(\frac{|1-\overline{z}\z|^{2}}{(1-|z|)(1-|\z|)}\right)^\frac{\delta}{q},\quad z,\z\in\D.
    \end{equation}

\smallskip

\noindent{\bf{Case $\mathbf{1}$}}. If
    $$
    \z\in D_1(z)=\left\{ w\in\D: 1-\frac{2|1-z\overline{w}|^2}{1-|z|^2}\le0\right\},
    $$
then $1-|z|\lesssim|1-z\overline{\z}|^2$ and
    \begin{equation*}
    \begin{split}
    \Gamma_\tau(z,\z)^q
    &\le\frac{\left(\frac{|1-\overline{z}\z|^2}{\max\{1-|z|^2,1-|\z|^2\}}\right)^{\frac{q}{p}-\tau-1}}
    {\min\left\{ \frac{\widehat{\nu}(z)}{(1-|z|)^\tau},
    \frac{\widehat{\nu}(\z)}{(1-|\z|)^\tau}\right\}}\widehat{\om}(0)^\frac{q}{p}
    \lesssim\left(\frac{|1-z\overline{\z}|^2}{1-|z|^2}\right)^{\frac{q}{p}-\tau-1}\frac{(1-|z|)^\tau}{\widehat{\nu}(z)}\chi_{D(0,|z|)}(\z)\\
    &\quad+\left(\frac{|1-z\overline{\z}|^2}{1-|z|^2}\right)^{\frac{q}{p}-\tau-1}\frac{(1-|\z|)^\tau}{\widehat{\nu}(\z)}\chi_{\D\setminus D(0,|z|)}(\z),\quad z\in\D,\quad \z\in D_1(z),
    \end{split}
    \end{equation*}
because of how $\tau$ is chosen in Lemma~\ref{bo}. Therefore \eqref{eq:delta} together with Lemmas~\ref{Lemma:weights-in-D-hat}(ii) and \ref{lemma:d-hat-new}(ii) yields
    \begin{equation*}
    \begin{split}
    &\frac{(1-|z|)^{c+1}\gamma(z)^q}{\widehat{\nu}(z)}\int_{D_1(z)} |(1+\beta(z,\z))\Gamma_\tau(z,\z)|^q \frac{(1-|\z|^2)^\sigma}{|1-z\overline{\z}|^{2+c+\sigma}}\nu(\z)\,dA(\z)\\
    &\lesssim\frac{(1-|z|)^{c+2+2\tau-\delta-\frac{q}{p}}\gamma(z)^q}{\widehat{\nu}(z)^{2}}
    \int_{D_1(z)\cap D(0,|z|)}\frac{(1-|\z|^2)^{\sigma-\delta}}{|1-z\overline{\z}|^{4+c+\sigma-2\left(\frac{q}{p}+\d-\tau\right)}}\nu(\z)\,dA(\z)\\
    &\quad+\frac{(1-|z|)^{c+2+\tau-\delta-\frac{q}{p}}\gamma(z)^q}{\widehat{\nu}(z)}
    \int_{D_1(z)\setminus D(0,|z|)} \frac{(1-|\z|^2)^{\sigma-\delta+\tau}}{\widehat{\nu}(\z)|1-z\overline{\z}|^{4+c+\sigma-2\left(\frac{q}{p}+\d-\tau\right)}}\nu(\z)\,dA(\z)\\
    &\lesssim\frac{(1-|z|)^{\frac{c}{2}+\tau-\frac{\sigma}{2}}\gamma(z)^q}{\widehat{\nu}(z)^{2}}
    \int_{0}^{|z|}(1-s)^{\sigma-\delta}\nu(s)\,ds\\
    &\quad+\frac{(1-|z|)^{\frac{c}{2}-\frac{\sigma}{2}}\gamma(z)^q}{\widehat{\nu}(z)}
    \int_{|z|}^1 (1-s)^{\sigma-\delta+\tau}\frac{\nu(s)}{\widehat{\nu}(s)}\,ds\\
    &\lesssim\frac{(1-|z|)^{\frac{c}{2}+\tau-\frac{\sigma}{2}+1-\frac{q}{p}}}{\widehat{\nu}(z)\widehat{\om}(z)^{\frac{q}{p}}}
    +\frac{(1-|z|)^{\frac{c}{2}+\frac{\sigma}{2}+1+\tau-\delta-\frac{q}{p}}}{\widehat{\om}(z)^{\frac{q}{p}}}\\
    &\lesssim(1-|z|)^{\frac{c}{2}+\tau-\frac{\sigma}{2}+1-\frac{q}{p}-\b(\nu)-\frac{q}{p}\b(\om)}
    \lesssim 1,\quad z\in\D,
    \end{split}
    \end{equation*}
where the last estimate is an immediate consequence of the choices of $c$ and $\delta$.

\smallskip

\noindent{\bf{Case $\mathbf{2}$}}. If
    $$
    \z\in D_2(z)=\left\{w\in\D: 1-\frac{2|1-z\overline{w}|^2}{1-|z|^2}\ge|z|\ge|w|\right\},
    $$
then $|1-z\overline{\z}|\asymp1-|z|^2\le1-|\z|^2$, which together the fact that $\frac{\widehat{\nu}(t)}{(1-t)^\tau}$ and $\frac{\widehat{\om}(r)}{(1-r)^{\tau\frac{p}{q}}}$ are essentially decreasing on $[0,1)$ gives
    \begin{equation*}
    \begin{split}
    \Gamma_\tau(z,\z)^q
    \lesssim\gamma(z)^{-q},\quad z\in\D,\quad \z\in D_2(z).
    \end{split}
    \end{equation*}
Therefore \eqref{eq:delta} and Lemma~\ref{Lemma:weights-in-D-hat}(iii) yield
    \begin{equation*}
    \begin{split}
    &\frac{(1-|z|)^{c+1}\gamma(z)^q}{\widehat{\nu}(z)}\int_{D_2(z)}|(1+\beta(z,\z))\Gamma_\tau(z,\z)|^q \frac{(1-|\z|^2)^\sigma}{|1-z\overline{\z}|^{2+c+\sigma}}\nu(\z)\,dA(\z)\\
    &\lesssim\frac{(1-|z|)^{c+1-\d}}{\widehat{\nu}(z)}
    \int_{D_2(z)}\frac{(1-|\z|^2)^{\sigma-\d}}{|1-z\overline{\z}|^{2+c+\sigma-2\d}}\nu(\z)\,dA(\z)\\
    &\lesssim\frac{(1-|z|)^{c+1-\d}}{\widehat{\nu}(z)}
    \int_0^{|z|}\frac{\nu(r)}{(1-r)^{c+1-\d}}\,dr
    \lesssim1,\quad z\in\D.
    \end{split}
    \end{equation*}

\smallskip

\noindent{\bf{Case $\mathbf{3}$}}. If
    $$
    \z\in D_3(z)=\left\{w\in\D:\min\left\{1-\frac{2|1-z\overline{w}|^2}{1-|z|^2},|w|\right\}\ge|z|\right\},
    $$
then $|1-z\overline{\z}|\asymp1-|z|^2\ge1-|\z|^2$, which together the fact that $\frac{\widehat{\nu}(t)}{(1-t)^\tau}$ and $\frac{\widehat{\om}(r)}{(1-r)^{\tau\frac{p}{q}}}$ are essentially decreasing on $[0,1)$ implies
    \begin{equation*}
    \begin{split}
    \Gamma_\tau(z,\z)^q\lesssim\frac{\widehat{\om}(z)^{\frac{q}{p}}(1-|z|)^{\frac{q}{p}-1}}{\widehat{\nu}(\z)},\quad z\in\D,\quad \z\in D_3(z).
    \end{split}
    \end{equation*}
Therefore \eqref{eq:delta} and Lemma~\ref{lemma:d-hat-new}(ii) implies
    \begin{equation*}
    \begin{split}
    &\frac{(1-|z|)^{c+1}\gamma(z)^q}{\widehat{\nu}(z)}\int_{D_3(z)}|(1+\beta(z,\z))\Gamma_\tau(z,\z)|^q \frac{(1-|\z|^2)^\sigma}{|1-z\overline{\z}|^{2+c+\sigma}}\nu(\z)\,dA(\z)\\
    &\lesssim(1-|z|)^{c+1-\d}\int_{D_3(z)}\frac{(1-|\z|^2)^{\sigma-\d}}{|1-z\overline{\z}|^{2+c+\sigma-2\delta}}
    \frac{\nu(\z)}{\widehat{\nu}(\z)} dA(\z)\\
    &\lesssim(1-|z|)^{\delta-\sigma}\int_{|z|}^1 \frac{(1-s)^{\sigma-\delta}\nu(s)}{\widehat{\nu}(s)}\,ds\lesssim 1,\quad z\in\D.
    \end{split}
    \end{equation*}

\smallskip

\noindent{\bf{Case $\mathbf{4}$}}.
If
    $$
    \z\in D_4(z)=\left\{w\in\D:1-\frac{2|1-z\overline{w}|^2}{1-|z|^2}<|z|\right\},
    $$
then Lemma~\ref{Lemma:weights-in-D-hat}(ii) gives
    $$
    \widehat{\om}\left(1-\frac{2|1-z\overline{\z}|^2}{1-|z|^2}\right)
    \lesssim\left(\frac{|1-z\overline{\z}|}{1-|z|}\right)^{2\beta(\om)}\widehat{\om}(z),\quad z\in\D,\quad \z\in D_4(z),
    $$
and hence
    \begin{equation*}
    \begin{split}
    \Gamma_\tau(z,\z)^q
    &\lesssim\left(\frac{|1-z\overline{\z}|}{1-|z|}\right)^{2\beta(\om)\frac{q}{p}}
    \widehat{\om}(z)^{\frac{q}{p}}
    \Bigg(\left(\frac{|1-z\overline{\z}|^2}{1-|\z|}\right)^{\frac{q}{p}-\tau-1}\frac{(1-|z|)^\tau}{\widehat{\nu}(z)}\chi_{D(0,|z|)}(\z)\\
    &\quad+\left(\frac{|1-z\overline{\z}|^2}{1-|z|}\right)^{\frac{q}{p}-\tau-1}\frac{(1-|\z|)^\tau}{\widehat{\nu}(\z)}\chi_{\D\setminus D(0,|z|)}(\z)
    \Bigg),\quad z\in\D,\quad \z\in D_4(z).
    \end{split}
    \end{equation*}
Hence \eqref{eq:delta} and Lemmas~\ref{Lemma:weights-in-D-hat}(iii) and \ref{lemma:d-hat-new}(ii) yield
    \begin{equation*}
    \begin{split}
    &\frac{(1-|z|)^{c+1}\gamma(z)^q}{\widehat{\nu}(z)}\int_{D_4(z)}|(1+\beta(z,\z))\Gamma_\tau(z,\z)|^q \frac{(1-|\z|^2)^\sigma}{|1-z\overline{\z}|^{2+c+\sigma}}\nu(\z)\,dA(\z)\\
    &\lesssim\frac{(1-|z|)^{c+2-\d-\frac{q}{p}-2\b(\om)\frac{q}{p}+\tau}}{\widehat{\nu}(z)}\int_{D_4(z)\cap D(0,|z|)}\frac{(1-|\z|^2)^{\sigma-\delta-\frac{q}{p}+\tau+1}}{|1-z\overline{\z}|^{4+c+\sigma-2\d-2\b(\om)\frac{q}{p}-2\frac{q}{p}+2\tau}}\nu(\z)\,dA(\z)\\
    &\quad+(1-|z|)^{c+2-\d-\frac{q}{p}-2\b(\om)\frac{q}{p}-\frac{q}{p}+\tau+1}
    \int_{D_4(z)\setminus D(0,|z|)}\frac{(1-|\z|)^{\s-\d+\tau}}{|1-z\overline{\z}|^{4+c+\s-2\d-2\b(\om)\frac{q}{p}-2\frac{q}{p}+2\tau}}\frac{\nu(\z)}{\widehat{\nu}(\z)}\,dA(\z)\\
    &\lesssim\frac{(1-|z|)^{c+2-\d-\frac{q}{p}-2\b(\om)\frac{q}{p}+\tau}}{\widehat{\nu}(z)}
    \int_0^{|z|}\frac{\nu(r)}{(1-r)^{2+c-\d-2\b(\om)\frac{q}{p}-\frac{q}{p}+\tau}}\,dr\\
    &\quad+\frac{1}{(1-|z|)^{\s-\d+\tau}}\int_{|z|}^1\frac{(1-r)^{\s-\d+\tau}\nu(r)}{\widehat{\nu}(r)}\,dr
    \lesssim1,\quad z\in\D.
    \end{split}
    \end{equation*}
Since $\D=\cup_{j=1}^4 D_j(z)$ for each $z\in\D$, by combining the four cases we obtain \eqref{eq:f1}. Thus (ii) implies (iii), and the proof is complete.
\end{proof}

\section{Boundedness of integral operators}
\label{Boundedness of integral operators}

In order to deal with the boundedness of Hankel operators, we need a technical result concerning certain integral operators. For $f\in L^1_{b}$ and $b,c\in\mathbb{R}$, define
    $$
    T_{b,c}(f)(z)=\int_\D f(\z)\frac{(1-|\z|^2)^b}{(1-z\overline{\z})^c}\,dA(\z),\quad z\in\D,
    $$
and
    $$
    S_{b,c}(f)(z)=\int_\D f(\z)\frac{(1-|\z|^2)^b}{|1-z\overline{\z}|^c}\,dA(\z),\quad z\in\D.
    $$
In the analytic case the operator $T_{b,c}$ can be interpreted as a fractional differentiation or integration depending on the parameters $b$ and $c$
\cite{ZZMF08}. The boundedness of these operator between $L^p$ spaces induced by standard weights has been characterized in \cite{ZhaoIEOT}.

Lemma~\ref{Lemma:weights-in-D-hat}(ii) shows that for $\eta\in\DD$ there exists a constant $c_0=c_0(\sigma)>1$ such that hypotheses (i) and (ii) of the next lemma are satisfied for all $c\ge c_0$.

\begin{lemma}\label{lemma:S}
Let $1<p\le q<\infty$, $b>-1$, $c>1$ and $\sigma,\eta\in\DDD$ such that
\begin{enumerate}
    \item[(i)] $\displaystyle\quad\int_r^1 \frac{(1-t)^{c-2}}{\widehat{\eta}(t)^{\frac1q}}\,dt\lesssim \frac{(1-r)^{c-1}}{\widehat{\eta}(r)^{\frac1q}},\quad 0\le r<1;$
    \item[(ii)] $\displaystyle\quad\int_0^r \frac{\eta(t)}{(1-t)^{\frac{cq}{p}-1}\widehat{\eta}(t)^{\frac1{p'}}}\,dt\lesssim\frac{\widehat{\eta}(r)^{\frac1p}}{(1-r)^{\frac{cq}{p}-1}},\quad 0\le r<1$.
\end{enumerate}
Then the following statements are equivalent:
\begin{enumerate}
\item $ S_{b,c}:A^p_\sigma\to L^q_\eta$ is bounded;
 \item $ T_{b,c}:A^p_\sigma\to L^q_\eta$ is bounded;
 \item $\displaystyle\sup_{0<r<1} (1-r)^{2+b-c+\frac{1}{q}-\frac{1}{p}}\frac{\widehat{\eta}(r)^{\frac1q}}{\widehat{\sigma}(r)^{\frac1p}}<\infty$.
\end{enumerate}
\end{lemma}

\begin{proof}
Obviously (1) implies (2). Assume now (2), and for each $\z\in\D$ and $N\in\N$, define $f_{\z,N}\in H^\infty$ by
    $
    f_{\z,N}(z)=\frac{z^N}{\sigma(S(\z))^{\frac{1}{p}}}\left(\frac{1-|\z|^2}{1-\overline{\z}z} \right)^{2+b+N}
    $
for all $z\in\D$. By differentiating the reproducing formula of $A^2_b$ we obtain
    \begin{equation}\label{eq:rpN}
    g^{(N)}(z)=M_1\int_\D\frac{\overline{u}^Ng(u)(1-|u|^2)^{b}}{(1-\overline{u}z)^{2+b+N}}\,dA(u),\quad z\in\D,\quad N\in\N,\quad g\in A^2_b,
    \end{equation}
where $M_1=M_1(N,b)>0$ is a constant. Therefore
    \begin{equation*}
    \begin{split}
    T_{b,c}(f_{\z,N})(z) &=\frac{(1-|\z|^2)^{2+b+N}}{\sigma(S(\z))^{\frac{1}{p}}}
    \int_{\D} \frac{u^N(1-|u|^2)^{b}}{(1-u\overline{\z})^{2+b+N}(1-\overline{u}z)^{c}}\,dA(u)\\
    &=\frac{(1-|\z|^2)^{2+b+N}}{\sigma(S(\z))^{\frac{1}{p}}}
    \overline{\int_{\D} \frac{\overline{u}^N(1-|u|^2)^{b}}{(1-\z\overline{u})^{2+b+N}(1-\overline{z}u)^{c}}\,dA(u)}\\
    &=M_2\frac{(1-|\z|^2)^{2+b+N}}{\sigma(S(\z))^{\frac{1}{p}}}\frac{z^N}{(1-z\overline{\z})^{c+N}},
    \end{split}
    \end{equation*}
where $M_2=M_2(b,c,N)>0$. Fix $N>\max\left\{ \frac{\lambda(\eta)+1}{q}-c,\frac{\lambda(\sigma)+1}{p}-b-2\right\}$. Then Lemma~\ref{Lemma:weights-in-D-hat}(iv) gives $\| f_{\z,N}\|_{L^p_\sigma}\asymp 1$ and
    $$
    \int_{\D} \frac{\eta(z)}{|1-\overline{\z}z|^{(c+N)q}}\,dA(z)\asymp \frac{\eta(S(\z))}{(1-|\z|)^{(c+N)q}},\quad \z\in\D.
    $$
Therefore (2) yields
    \begin{equation*}
    \begin{split}
    \infty
    &>\| f_{\z,N}\|^q_{L^p_\sigma}
    \gtrsim\| T_{b,c}(f_{\z,N})\|^q_{L^q_\eta}
    \asymp\left(\frac{(1-|\z|^2)^{2+b+N}}{\sigma(S(\z))^{\frac{1}{p}}}\right)^{q}\int_{\D} \frac{\eta(z)}{|1-\overline{\z}z|^{(c+N)q}}\,dA(z)\\
    &\asymp(1-|\z|^2)^{q(2+b-c)}\frac{\eta(S(\z))}{\sigma(S(\z))^{\frac{q}{p}}},\quad \z\in\D,
    \end{split}
    \end{equation*}
thus (3) holds.

Assume (3) holds and let $h(\zeta)=\widehat{\sigma}(\zeta)^{\frac{1}{pp'}}(1-|\zeta|^2)^{\frac{b}{p}+\left(\frac{1}{p}-\frac{1}{q}\right)\frac{1}{p'}}$ for all $\z\in\D$. Then H\"older's inequality yields
    \begin{align*}
    |S_{b,c}f(z)|
    &\le\left(\int_\D |f(\zeta)|^p h(\zeta)^p\frac{dA(\zeta)}{|1-z\overline{\zeta}|^c}\right)^{\frac{1}{p}}\left(\int_\D \left(\frac{(1-|\zeta|^2)^b}{h(\zeta)}\right)^{p'}\frac{dA(\zeta)}{|1-z\overline{\zeta}|^c}\right)^{\frac{1}{p'}}
    =I_1(z)^{\frac{1}{p}}\cdot I_2(z)^{\frac{1}{p'}},
    \end{align*}
where
    \begin{align*}
    I_2(z)
    &=\int_\D\frac{(1-|\zeta|^2)^{b-\frac{1}{p}+\frac{1}{q}}}{|1-z\overline{\zeta}|^c\widehat{\sigma}(\zeta)^{\frac{1}{p}}}\,dA(\zeta)
    \asymp \int_0^1 \frac{(1-r)^{b-\frac{1}{p}+\frac{1}{q}}}{\widehat{\sigma}(r)^{\frac{1}{p}}(1-r|z|)^{c-1}}\,dr\\
    &=\int_0^{|z|} \frac{(1-r)^{b-\frac{1}{p}+\frac{1}{q}}}{\widehat{\sigma}(r)^{\frac{1}{p}}(1-r|z|)^{c-1}}\,dr+\int_{|z|}^1 \frac{(1-r)^{b-\frac{1}{p}+\frac{1}{q}}}{\widehat{\sigma}(r)^{\frac{1}{p}}(1-r|z|)^{c-1}}\,dr
    =J^{|z|}+J_{|z|}.
    \end{align*}
Lemma~\ref{Lemma:weights-in-D} together with the assumption (3) yields
    $$
    J^{|z|}
    \le\int_0^{|z|}\frac{(1-r)^{b-\frac1p+\frac1q+1-c}}{\widehat{\sigma}(r)^\frac1p}\,dr
    \lesssim\int_0^{|z|}\frac{dr}{\widehat{\eta}(r)^{\frac{1}{q}}(1-r)}\lesssim \frac{1}{\widehat{\eta}(z)^{\frac{1}{q}}},\quad z\in\D,
    $$
since $\eta\in\DDD\subset\Dd$ by the hypothesis. In a similar fashion, (3) together with the hypothesis (i) gives
    $$
    J_{|z|}
    \le\frac1{(1-|z|)^{c-1}}\int_{|z|}^1\frac{(1-r)^{b-\frac1p+\frac1q}}{\widehat{\sigma}(r)^\frac1p}\,dr
    \lesssim\frac{1}{(1-|z|)^{c-1}}\int_{|z|}^1 \frac{(1-r)^{c-2}}{\widehat{\eta}(r)^{\frac{1}{q}}}\,dr\lesssim \frac{1}{\widehat{\eta}(z)^{\frac{1}{q}}},\quad z\in\D,
    $$
and hence $I_2(z)\lesssim\widehat{\eta}(z)^{-\frac1q}$ for all $z\in\D$. This estimate and Minkowski's integral inequality (Fubini's theorem in the case $q=p$) now yield
    \begin{align*}
    \|S_{b,c}(f)\|_{L^q_\eta}^p
    &\lesssim\left(\int_\D\left(\int_\D|f(\zeta)|^p h(\zeta)^p\frac{dA(\z)}{|1-z\overline{\zeta}|^c}\right)^{\frac{q}{p}}\frac{\eta(z)}{\widehat{\eta}(z)^{\frac{1}{p'}}}dA(z)\right)^{\frac{p}{q}}
    \le\int_\D |f(\zeta)|^p\widetilde{\sigma}(\zeta)I_3(\zeta)\,dA(\zeta),
    \end{align*}
where
    $$
    I_3(\zeta)
    =\frac{h(\zeta)^p}{\widetilde{\sigma}(\zeta)}
    \left(\int_\D\frac{\eta(z)dA(z)}{|1-z\overline{\zeta}|^{\frac{cq}{p}}\widehat{\eta}(z)^{\frac1{p'}}}\right)^{\frac{p}{q}}
    \asymp\frac{h(\zeta)^p}{\widetilde{\sigma}(\zeta)}
    \left(\int_0^1\frac{\eta(r)}{(1-r|\zeta|)^{\frac{cq}{p}-1}\widehat{\eta}(r)^{\frac1{p'}}}\,dr\right)^{\frac{p}{q}}.
    $$
Since
    \begin{align*}
    \int_0^{|\z|}\frac{\eta(r)}{(1-r|\zeta|)^{\frac{cq}{p}-1}\widehat{\eta}(r)^{1/p'}}\,dr
    &\le \int_0^{|\z|}\frac{\eta(r)}{(1-r)^{\frac{cq}{p}-1}\widehat{\eta}(r)^{\frac1{p'}}}\,dr
    \lesssim\frac{\widehat{\eta}(\z)^\frac1p}{(1-|\z|)^{\frac{cq}{p}-1}},\quad \z\in\D,
    \end{align*}
by the hypothesis (ii), and
    \begin{align*}
    \int_{|\z|}^1\frac{\eta(r)}{(1-r|\zeta|)^{\frac{cq}{p}-1}\widehat{\eta}(r)^{\frac1{p'}}}\,dr
    \le\frac{1}{(1-|\zeta|)^{\frac{cq}{p}-1}}\int_{|\z|}^{1}\frac{\eta(r)}{\widehat{\eta}(r)^{\frac1{p'}}}\,dr
    \asymp\frac{\widehat{\eta}(\z)^\frac1p}{(1-|\z|)^{\frac{cq}{p}-1}},\quad \z\in\D,
    \end{align*}
we deduce
    $$
    I_3(\zeta)\lesssim (1-|\zeta|)^{2+b-c+\frac1q-\frac1p}\frac{\widehat{\eta}(\zeta)^{\frac1q}}{\widehat{\sigma}(\zeta)^{\frac1p}}\lesssim 1,\quad \z\in\D,
    $$
by the assumption (3). It follows that $\|S_{b,c}(f)\|_{L^q_\eta}\lesssim\|f\|_{A^p_{\widetilde{\sigma}}}$. This finishes the proof because $\|f\|_{A^p_{\widetilde{\sigma}}}\asymp\|f\|_{A^p_{{\sigma}}}$ for all $f\in\H(\D)$ by Lemma~\ref{LemmaDeMierda} provided $\sigma\in\DDD$.
\end{proof}

\section{Proof of Theorem~\ref{th:Hankelqbiggerp}}
\label{Proof of main theorem}

In order to prove the sufficiency part of Theorem~\ref{th:Hankelqbiggerp} we shall use the next result which follows from the argument used in the proof of \cite[Lemma~4.5]{PZZ}.

\begin{lemma}\label{le:bq}
Let $1<q<\infty$ and $\nu,\om$ weights such that $P_\om: L^q_\nu\to L^q_\nu$ is bounded. Then
    $$
    \|H^\nu_f(g)\|^q_{L^q_\nu}\le (1+\|P_\om\|_{L^q_\nu\to L^q_\nu})\|H^\om_f(g)\|^q_{L^q_\nu},\quad f\in L^q_\nu,\quad g\in H^\infty.
    $$
\end{lemma}

\begin{proposition}\label{pr:SufBopart}
Let $1<p\le q<\infty$, $\nu\in B_q$ a radial weight and $\om\in\DDD$. If $f\in\BO(\Delta)$, then $H^\nu_f:A^p_\om\to L^q_\nu$ is bounded.
\end{proposition}

\begin{proof}
By \cite{BB} there exists a constant $s_0=s_0(\nu)>-1$ such that $P_s:L^q_\nu\to L^q_\nu$ is bounded for each $s>s_0$. Let $0<\tau<\min\{q\a(\om)/p, \a(\nu)\}$, where $\alpha(\nu)$ and $\a(\om)$ are those from Lemma~\ref{Lemma:weights-in-D}. Then Lemmas~\ref{bo} and~\ref{le:bq} yield
    \begin{equation*}
    \begin{split}
    \|H^\nu_f(g)\|_{L^q_\nu}^q
    &\lesssim\|H^s_f(g)\|_{L^q_\nu}^q
    \le\int_\D\left(\int_\D\frac{|f(z)-f(\z)||g(\z)|}{|1-\overline{z}\z|^{2+s}}(1-|\z|^2)^s\,dA(\z)\right)^q\nu(z)\,dA(z)\\
    &\lesssim\int_\D\left(\int_\D|g(\z)|\frac{(\b(z,\z)+1)\G_\tau(z,\z)}{|1-\overline{z}\z|^{2+s}}(1-|\z|^2)^s\,dA(\z)\right)^q\nu(z)\,dA(z),\quad g\in H^\infty.
    \end{split}
    \end{equation*}
Let $s>\max\left\{s_0,2\left(\b(\om)+\b(\nu)+2\a(\nu)\right)\right\}$, $\d<\min\{\frac{\tau}{q},\frac{\alpha(\nu)}{q}\}$ and $K>1$ to be fixed later.
Then applying \eqref{eq:delta}, we get
   \begin{equation}
    \begin{split}\label{eq:suf1}
    \|H^\nu_f(g)\|_{L^q_\nu}^q
    &\lesssim
    \sum_{j=1}^5\int_\D\left(\int_{\Omega_j(z)}|g(\z)|\frac{\G_\tau(z,\z)\,dA(\z)}{|1-\overline{z}\z|^{2+s-2\d}(1-|\z|^2)^{\d-s}}\right)^q
    \frac{\nu(z)}{(1-|z|)^{q\d}}\,dA(z)\\
    &=\sum_{j=1}^5I_j(g),
    \end{split}
    \end{equation}
where
    \begin{equation*}
    \begin{split}
    \Omega_1(z)&=\left\{\z\in\D:\frac{1}{|1-\overline{z}\z|^2}\le\frac{2}{\max\{1-|z|^2,1-|\z|^2\}}\right\}\cap D(0,|z|),\\ \Omega_2(z)&=\left\{\z\in\D:\frac{1}{|1-\overline{z}\z|^2}\le\frac{2}{\max\{1-|z|^2,1-|\z|^2\}}\right\}\cap \left(\D\setminus D(0,|z|)\right),\\
    \Omega_3(z,K)&=\left\{\z\in\D:
      \frac{1-|\z|}{K}\ge\frac{2|1-\overline{z}\z|^2}{\max\{1-|z|^2,1-|\z|^2\}}\right\},\\
        \Omega_4(z,K)&=
                    \left\{\z\in\D:\frac{1-|\z|}{K}<\frac{2|1-\overline{z}\z|^2}{\max\{1-|z|^2,1-|\z|^2\}}<1\right\}
                    \cap D(0,|z|),\\
        \Omega_5(z,K)&=
         \left\{\z\in\D:\frac{1-|\z|}{K}<\frac{2|1-\overline{z}\z|^2}{\max\{1-|z|^2,1-|\z|^2\}}<1\right\}
                    \cap(\D\setminus D(0,|z|)).
    \end{split}
    \end{equation*}
The quantities $I_j(g)$, $j=1,\dots,5$, will be estimated separately.

\smallskip

\noindent{\bf{Case $\mathbf{I_1(g)}$}.} By using the definition of $\Omega_1(z)$, and the fact that $\frac{\widehat{\nu}(x)}{(1-x)^\tau}$ is essentially decreasing on $[0,1)$ we deduce
    \begin{equation*}
    \begin{split}
    \G_\tau(z,\z)
    \lesssim\frac{\left(\frac{|1-\overline{z}\z|^2}{\max\{1-|z|^2,1-|\z|^2\}}\right)^{\frac{1}{p}-\frac{1}{q}}}
    {\min\left\{ \frac{\widehat{\nu}(z)}{(1-|z|)^\tau},\frac{\widehat{\nu}(\z)}{(1-|\z|)^\tau}\right\}^{\frac1q}}
    \lesssim\left(\frac{|1-\overline{z}\z|^2}{1-|\z|^2} \right)^{\frac{1}{p}-\frac{1}{q}}
    \left(\frac{(1-|z|)^\tau}{\widehat{\nu}(z)}\right)^{\frac1q},\quad z\in\D,\quad \z\in\Omega_1(z).
    \end{split}
    \end{equation*}
Then the estimate
    \begin{equation}\label{9}
    M_1(r,f)\le M_p(r,f)\lesssim\|f\|_{A^p_\om}\widehat{\om}(r)^{-\frac1p},\quad 0\le r<1,\quad f\in\H(\D),
    \end{equation}
and Lemma~\ref{lemma:d-hat-new}(ii) yield
    \begin{equation*}
    \begin{split}
    I_1(g)
    &\lesssim\int_\D\left(\int_{\Omega_1(z)}|g(\z)|\frac{(1-|\z|)^{s-\d-\frac1p+\frac1q}}{|1-\overline{z}\z|^{2+s-2\d-2\left(\frac1p-\frac1q\right)}}\,dA(\z)\right)^q
    \frac{\nu(z)(1-|z|)^{\tau-\d q}}{\widehat{\nu}(z)}\,dA(z)\\
    &\lesssim\left(\int_\D|g(\z)|(1-|\z|)^{\frac{s}2-1}\,dA(\z)\right)^q\int_\D\frac{\nu(z)(1-|z|)^{\tau-\d q}}{\widehat{\nu}(z)}\,dA(z)\\
    &\lesssim\|g\|_{A^p_{\om}}^q
    \left(\int_{0}^{1}\frac{(1-t)^{\frac{s}{2}-1}}{\widehat{\om}(t)^{\frac1p}}\,dt\right)^q
    \lesssim\|g\|_{A^p_{\om}}^q
    \left( \int_{0}^{1}(1-t)^{\frac{s}{2}-1-\frac{\b(\om)}{p}}\,dt\right)^q\lesssim\|g\|^q_{A^p_{\om}},\quad g\in H^\infty.
    \end{split}
    \end{equation*}

\smallskip

\noindent{\bf{Case $\mathbf{I_2(g)}$}.} The definition of $\Omega_2(z)$ and the monotonicity of $\frac{\widehat{\nu}(x)}{(1-x)^\tau}$ imply
    \begin{equation*}
    \begin{split}
    \G_\tau(z,\z) \lesssim \frac{\left(\frac{|1-\overline{z}\z|^2}{\max\{1-|z|^2,1-|\z|^2\}}\right)^{\frac{1}{p}-\frac{1}{q}}}
    {\min\left\{ \frac{\widehat{\nu}(z)}{(1-|z|)^\tau},
    \frac{\widehat{\nu}(\z)}{(1-|\z|)^\tau}\right\}^{1/q}}
    \lesssim \left( \frac{|1-\overline{z}\z|^2}{1-|\z|^2} \right)^{\frac{1}{p}-\frac{1}{q}}
    \left(\frac{(1-|\z|)^\tau}{\widehat{\nu}(\z)}\right)^{\frac1q},\quad z\in\D,\quad \z\in\Omega_2(z).
    \end{split}
    \end{equation*}
Therefore \eqref{9} and Lemmas~\ref{Lemma:weights-in-D-hat} and~\ref{Lemma:weights-in-D} yield
    \begin{equation*}
    \begin{split}
    I_2(g)
    &\lesssim
    \int_\D\left(\int_{\Omega_2(z)}|g(\z)|
    \frac{(1-|\z|)^{s-\d-\frac1p+\frac{1+\tau}{q}}}{\widehat{\nu}(\z)^{\frac1q}|1-\overline{z}\z|^{2+s-2\d-2\left(\frac1p-\frac1q\right)}}\,dA(\z)\right)^q
    (1-|z|)^{-\d q}\nu(z)\,dA(z)\\
    &\lesssim\left(\int_\D|g(\z)|\frac{(1-|\z|)^{\frac{s}2-1+\frac{\tau}{q}}}
    {\widehat{\nu}(\z)^{\frac1q}}\,dA(\z)\right)^q\int_0^1(1-r)^{-\d q}\nu(r)\,dr\\
    &\lesssim\|g\|_{A^p_{\om}}^q\left(\int_0^1\frac{(1-r)^{\frac{s}{2}-1+\frac{\tau}{q}}}
    {\widehat{\om}(r)^\frac{1}{p}\widehat{\nu}(\z)^{\frac1q}}\,dA(\z)\right)^q
    \left(\widehat{\nu}(0)+ \int_0^1 \frac{\widehat{\nu}(t)}{(1-t)^{1+q\d}}\,dt\right)\\
    & \lesssim\|g\|_{A^p_{\om}}^q
    \left( \int_{0}^{1}(1-t)^{\frac{s}{2}-1-\frac{\b(\om)}{p}-\frac{\b(\nu)-\tau}{q}}\,dt\right)^q
    \lesssim\|g\|_{A^p_{\om}}^q,\quad g\in H^\infty.
    \end{split}
    \end{equation*}

\noindent{\bf{Case $\mathbf{I_3(g)}$}.} To deal with $I_3(g)$, note first that now $2K|1-\overline{z}\z|^2\le(1-|\z|)\max\{1-|z|^2,1-|\z|^2\}\le2\left(\max\{1-|z|,1-|\z|\}\right)^2$ for all $\z\in\Omega_3(z,K)$. Hence $\z\in\Delta(z,R)$ for some $R=R(K)\in(0,\infty)$ if $K\ge 1$ is sufficiently large. Fix such a $K$, and note that then $\widehat{\nu}(\z)\asymp\widehat{\nu}(z)$ for all $\z\in\Omega(z,K)$ by Lemma~\ref{Lemma:weights-in-D-hat}(ii). By using this and the fact that $\frac{\widehat{\om}(x)}{(1-x)^{\frac{p\tau}{q}}}$ is essentially decreasing on $[0,1)$ we deduce
    \begin{equation*}
    \begin{split}
    \G_\tau(z,\z)
    &\lesssim\left(\frac{|1-\overline{z}\z|^2}{\max\{1-|z|^2,1-|\z|^2\}}\right)^{\frac{1}{p}-\frac{1}{q}} \frac{\widehat{\om}(\z)^{\frac1p}}{(1-|\z|)^{\frac{\tau}{q}}}
    {\min\left\{ \frac{\widehat{\nu}(z)}{(1-|z|)^\tau},
    \frac{\widehat{\nu}(\z)}{(1-|\z|)^\tau}\right\}^{-\frac1q}}\\
    &\asymp
    \frac{\left(1-|\z|\right)^{\frac{1}{p}-\frac{1}{q}}\widehat{\om}(\z)^{\frac1p}}{\widehat{\nu}(\z)^{\frac1q}},\quad z\in\D,\quad \z\in\Omega_3(z,K),    \end{split}
    \end{equation*}
and it follows that
    \begin{equation}
    \begin{split}\label{eq:j2}
    I_3(g)&\lesssim
    \int_\D\left(\int_{\Delta(z,R)}\left(|g(\z)|\widehat{\om}(\z)^\frac1p\right)
    \frac{\left( 1-|\z|^2 \right)^{s-\d+\frac{1}{p}-\frac{2}{q}}dA(\z)}{|1-\overline{z}\z|^{2+s-2\d}}\right)^q\frac{\nu(z)(1-|z|)^{1-q\d}}{\widehat{\nu}(z)}\,dA(z)
    \\ & \asymp
    \int_\D\left(\int_{\Delta(z,R)}\left(|g(\z)|\widetilde{\om}(\z)^\frac1p\right)
    \frac{\left( 1-|\z|^2 \right)^{s+\frac{2}{p}-\frac{2}{q}}dA(\z)}{|1-\overline{z}\z|^{2+s}}\right)^q\frac{\nu(z)(1-|z|)}{\widehat{\nu}(z)}\,dA(z)
    \\ & \asymp \int_\D\left(\int_{\Delta(z,R)}\left(|g(\z)|\widetilde{\om}(\z)^\frac1p\right)
    \frac{\left( 1-|\z|^2 \right)^{s}dA(\z)}{|1-\overline{z}\z|^{2+s-\frac{2}{p}+\frac{2}{q}}}\right)^q\frac{\nu(z)(1-|z|)}{\widehat{\nu}(z)}\,dA(z)
    \\ & \le \int_\D\left(\int_{\D}\left(|g(\z)|\widetilde{\om}(\z)^\frac1p\right)
    \frac{\left( 1-|\z|^2 \right)^{s}dA(\z)}{|1-\overline{z}\z|^{2+s-\frac{2}{p}+\frac{2}{q}}}\right)^q\frac{\nu(z)(1-|z|)}{\widehat{\nu}(z)}\,dA(z)\\
    &=\left\|S_{s,s+2\left(1-\frac1p+\frac1q\right)}\left(|g|\widetilde{\omega}^\frac1p\right)\right\|_{L^q_\eta}^q
    =\left\|S_{b,c}\left(|g|\widetilde{\omega}^\frac1p\right)\right\|_{L^q_\eta}^q,\quad g\in H^\infty,
    \end{split}
    \end{equation}
where $\eta(z)=\frac{\nu(z)(1-|z|)}{\widehat{\nu}(z)}$ for all $z\in\D$. To apply Lemma~\ref{lemma:S} with $\sigma\equiv1$, we must check that its hypotheses are satisfied. To do this, first observe that $\eta\in\DDD$ and $\widehat{\eta}(r)\asymp(1-r)$ for all $0\le r<1$ by Lemma~\ref{le:nuhiyperbolic}. Hence
    \begin{equation*}
    \begin{split}
    \int_r^1\frac{(1-t)^{c-2}}{\widehat{\eta}(t)^{\frac1q}}\,dt
    \asymp\int_r^1(1-t)^{s-\frac{2}{p}+\frac{1}{q}}\,dt
    \asymp(1-r)^{1+s-\frac{2}{p}+\frac{1}{q}}
    \asymp\frac{(1-r)^{c-1}}{\widehat{\eta}(r)^{\frac1q}},\quad 0\le r<1,
    \end{split}
    \end{equation*}
and, by Lemma~\ref{lemma:d-hat-new}(iii),
    \begin{equation*}
    \begin{split}
    \int_0^r\frac{\eta(t)}{(1-t)^{\frac{cq}{p}-1}\widehat{\eta}(t)^{\frac1{p'}}}\,dt
    &\asymp\int_0^r\frac{\nu(t)}{\widehat{\nu}(t)
    (1-t)^{\frac{q}{p}\left(s+2\left(1-\frac1p+\frac1q\right)\right)-1-\frac{1}{p}}}\,dt\\
    &\lesssim\frac{1}{(1-r)^{\frac{q}{p}\left(s+2\left(1-\frac1p+\frac1q\right)\right)-1-\frac{1}{p}}}
    \asymp\frac{\widehat{\eta}(r)^{\frac1p}}{(1-r)^{\frac{cq}{p}-1}},\quad 0\le r<1,
    \end{split}
    \end{equation*}
so the hypotheses of Lemma~\ref{lemma:S} are satisfied. Moreover,
    \begin{equation*}
    \begin{split}
    (1-r)^{2+b-c+\frac1q-\frac1p}\frac{\widehat{\eta}(r)^{\frac1q}}{\widehat{\sigma}(r)^{\frac1p}}\asymp 1,\quad 0\le r<1,
    \end{split}
    \end{equation*}
and consequently \eqref{eq:j2} and Lemmas~\ref{lemma:S} and~\ref{LemmaDeMierda} yield $I_3(g)\lesssim\|g\|_{A^p_{\widetilde{\om}}}^q\asymp\|g\|_{A^p_{\om}}^q$ for all $g\in H^\infty$.

\smallskip

\noindent{\bf{Case $\mathbf{I_4(g)}$}.} By using the definition of $\Omega_4(z,K)$, Lemma~\ref{Lemma:weights-in-D-hat}(ii) and the fact that $\frac{\widehat{\nu}(x)}{(1-x)^\tau}$ is essentially decreasing on $[0,1)$, we deduce
    \begin{equation*}
    \begin{split}
    \G_\tau(z,\z)
    &\lesssim\frac{\left(\frac{|1-\overline{z}\z|^2}{\max\{1-|z|^2,1-|\z|^2\}}\right)^{\frac{1}{p}-\frac{1}{q}}
    \widehat{\om}\left(1-\frac{2|1-\overline{z}\z|^2}{\max\{1-|z|^2,1-|\z|^2\}}\right)^{\frac{1}{p}}}
    {(1-|\z|)^{\frac{\tau}{q}}\min\left\{\frac{\widehat{\nu}(z)}{(1-|z|)^\tau},\frac{\widehat{\nu}(\z)}{(1-|\z|)^\tau}\right\}^{\frac1q}}\\
    &\lesssim\frac{\left(\frac{|1-\overline{z}\z|^2}{1-|\z|}\right)^{\frac{1}{p}-\frac{1}{q}}\widehat{\om}\left(1-\frac{K2|1-\overline{z}\z|^2}{\max\{1-|z|^2,1-|\z|^2\}}\right)^{\frac{1}{p}}}
    {(1-|\z|)^{\frac{\tau}{q}}\min\left\{\frac{\widehat{\nu}(z)}{(1-|z|)^\tau},\frac{\widehat{\nu}(\z)}{(1-|\z|)^\tau}\right\}^{\frac1q}}
    \lesssim\frac{\left(\frac{|1-\overline{z}\z|^2}{1-|\z|}\right)^{\frac{1}{p}-\frac{1}{q}}\left(\frac{|1-\overline{z}\z|}{1-|\z|}\right)^{\frac{2\beta(\om)}{p}}\widehat{\om}(\z)^{\frac{1}{p}}}
    {(1-|\z|)^{\frac{\tau}{q}}\min\left\{\frac{\widehat{\nu}(z)}{(1-|z|)^\tau},\frac{\widehat{\nu}(\z)}{(1-|\z|)^\tau}\right\}^{\frac1q}}\\
    &\lesssim\frac{|1-\overline{z}\z|^{\frac{2\beta(\om)}{p}+\frac{2}{p}-\frac{2}{q}}}{(1-|\z|)^{\frac{2\beta(\om)}{p}+\frac{1}{p}-\frac{1}{q}}}
    \frac{\widehat{\om}(\z)^{\frac{1}{p}}}
    {(1-|\z|)^{\frac{\tau}{q}}}\left(\frac{(1-|z|)^\tau}{\widehat{\nu}(z)}\right)^{\frac1q},\quad z\in\D,\quad \z\in\Omega_4(z,K).
    \end{split}
    \end{equation*}
Therefore
    \begin{equation}
    \begin{split}\label{I4}
    I_4(g)&\lesssim\int_\D\left(\int_{\D}\left(|g(\z)|\widetilde{\om}(\z)^\frac1p\right)
    \frac{(1-|\z|)^{s-\d-\frac{2\b(\om)}{p}+\frac1q-\frac{\tau}{q}}}
    {|1-\overline{z}\z|^{2+s-2\d-\frac{2\b(\om)}{p}-\frac{2}{p}+\frac{2}{q}}}dA(\z)\right)^q\frac{\nu(z)(1-|z|)^{\tau-\d q}}{\widehat{\nu}(z)}\,dA(z)\\
    &=\left\|S_{b,c}\left(|g|\widetilde{\omega}^\frac1p\right)\right\|_{L^q_\eta}^q,\quad g\in H^\infty,
    \end{split}
    \end{equation}
where $b=s-\d-\frac{2\b(\om)}{p}+\frac1q-\frac{\tau}{q}$, $c=2+s-2\d-\frac{2\b(\om)}{p}-\frac{2}{p}+\frac{2}{q}$ and $\eta(z)=\frac{\nu(z)(1-|z|)^{\tau-\d q}}{\widehat{\nu}(z)}$ for all $z\in\D$. We will appeal to Lemma~\ref{lemma:S} with $\sigma\equiv1$. First observe that $\eta\in\DDD$ and $\widehat{\eta}(r)\asymp(1-r)^{\tau-\d q}$ for all $0\le r<1$ by Lemma~\ref{le:nuhiyperbolic}. Hence
    \begin{equation*}
    \begin{split}
    \int_r^1\frac{(1-t)^{c-2}}{\widehat{\eta}(t)^{\frac1q}}\,dt
    &\asymp\int_r^1 (1-t)^{s-\d-\frac{2\beta(\om)}{p}-\frac{\tau}{q}-\frac{2}{p}+\frac{2}{q}}\,dt
    \asymp(1-r)^{1+s-\d-\frac{2\beta(\om)}{p}-\frac{\tau}{q}-\frac{2}{p}+\frac{2}{q}}\\
    &\asymp\frac{(1-r)^{c-1}}{\widehat{\eta}(r)^{\frac1q}},\quad 0\le r<1,
    \end{split}
    \end{equation*}
and, by Lemma~\ref{lemma:d-hat-new}(iii),
    \begin{equation*}
    \begin{split}
    \int_0^r\frac{\eta(t)}{(1-t)^{\frac{cq}{p}-1}\widehat{\eta}(t)^{\frac1{p'}}}\,dt
    &\asymp\int_0^r\frac{\nu(t)}{\widehat{\nu}(t)(1-t)^{\frac{q}{p}\left(2+s-2\d-\frac{2\b(\om)}{p}-\frac{2}{p}+\frac{2}{q}\right)-1-\frac{\tau-q\d}{p}}} \, dt\\
    &\lesssim\frac{1}{(1-r)^{\frac{q}{p}\left(2+s-2\d-\frac{2\b(\om)}{p}-\frac{2}{p}+\frac{2}{q}\right)-1-\frac{\tau-q\d}{p}}}
    \asymp\frac{\widehat{\eta}(r)^{\frac1p}}{(1-r)^{\frac{cq}{p}-1}},\quad 0\le r<1,
    \end{split}
    \end{equation*}
so the hypotheses of Lemma~\ref{lemma:S} are satisfied. Moreover,
    \begin{equation*}
    \begin{split}
    (1-r)^{2+b-c+\frac1q-\frac1p}\frac{\widehat{\eta}(r)^{\frac1q}}{\widehat{\sigma}(r)^{\frac1p}}
    \asymp1,\quad 0\le r<1,
    \end{split}
    \end{equation*}
and hence \eqref{I4} and Lemmas~\ref{lemma:S} and~\ref{LemmaDeMierda} imply $I_4(g)\lesssim\|g\|_{A^p_{\widetilde{\om}}}^q\asymp\|g\|_{A^p_{\om}}^q$ for all $g\in H^\infty$.

\smallskip

\noindent{\bf{Case $\mathbf{I_5(g)}$}.} By using the definition of $\Omega_5(z,K)$, Lemma~\ref{Lemma:weights-in-D-hat}(ii) and the fact that $\frac{\widehat{\nu}(x)}{(1-x)^\tau}$ is essentially decreasing on $[0,1)$ we deduce
    \begin{equation*}
    \begin{split}
    \G_\tau(z,\z)
    &\lesssim\frac{\left(\frac{|1-\overline{z}\z|^2}{\max\{1-|z|^2,1-|\z|^2\}}\right)^{\frac{1}{p}-\frac{1}{q}}
    \widehat{\om}\left(1-\frac{2|1-\overline{z}\z|^2}{\max\{1-|z|^2,1-|\z|^2\}}\right)^{\frac{1}{p}}}
    {(1-|\z|)^{\frac{\tau}{q}}\min\left\{ \frac{\widehat{\nu}(z)}{(1-|z|)^\tau},
    \frac{\widehat{\nu}(\z)}{(1-|\z|)^\tau}\right\}^{\frac1q}}\\
    &\lesssim\frac{\left(\frac{|1-\overline{z}\z|^2}{1-|\z|}\right)^{\frac{1}{p}-\frac{1}{q}}
    \widehat{\om}\left(1-\frac{K2|1-\overline{z}\z|^2}{\max\{1-|z|^2,1-|\z|^2\}}\right)^{\frac{1}{p}}}
    {(1-|\z|)^{\frac{\tau}{q}}\min\left\{ \frac{\widehat{\nu}(z)}{(1-|z|)^\tau},
    \frac{\widehat{\nu}(\z)}{(1-|\z|)^\tau}\right\}^{\frac1q}}
    \lesssim\frac{\left(\frac{|1-\overline{z}\z|^2}{1-|\z|}\right)^{\frac{1}{p}-\frac{1}{q}}
    \left(\frac{|1-\overline{z}\z|}{1-|\z|^2}\right)^{\frac{2\beta(\om)}{p}}
    \widehat{\om}(\z)^{\frac{1}{p}}}
    {(1-|\z|)^{\frac{\tau}{q}}\min\left\{ \frac{\widehat{\nu}(z)}{(1-|z|)^\tau},
    \frac{\widehat{\nu}(\z)}{(1-|\z|)^\tau}\right\}^{\frac1q}}\\
    &\lesssim\left(\frac{|1-\overline{z}\z|^2}{1-|\z|}\right)^{\frac{1}{p}-\frac{1}{q}}
    \left(\frac{|1-\overline{z}\z|}{1-|\z|}\right)^{\frac{2\beta(\om)}{p}}
    \frac{\widehat{\om}(\z)^{\frac{1}{p}}}
    {\widehat{\nu}(\z)^{\frac1q}},\quad z\in\D,\quad \z\in\Omega_5(z,K).
    \end{split}
    \end{equation*}
Therefore Lemma~\ref{Lemma:weights-in-D-hat}(ii) yields
    \begin{equation}
    \begin{split}\label{I5}
    I_5(g)
    &\lesssim\int_\D\left(\int_{\Omega_5(z,K)}\left(|g(\z)|\widetilde{\om}(\z)^\frac1p\right)
    \frac{(1-|\z|)^{s-\d-\frac{2\b(\om)}{p}+\frac1q}}
    {\widehat{\nu}(\z)^{\frac{1}{q}}|1-\overline{z}\z|^{2+s-2\d-\frac{2\b(\om)}{p}-\frac2p+\frac2q}}dA(\z)\right)^q\frac{\nu(z)\,dA(z)}{(1-|z|)^{q\d}}\\
    &\lesssim \int_\D\left(\int_{\D}\left(|g(\z)|\widetilde{\om}(\z)^\frac1p\right)
    \frac{(1-|\z|)^{s-\d-\frac{2\b(\om)}{p}+\frac1q-\frac{\b(\nu)}{q}}}
    {|1-\overline{z}\z|^{2+s-2\d-\frac{2\b(\om)}{p}-\frac2p+\frac2q}}dA(\z)\right)^q
    \frac{\nu(z)\,dA(z)}{(1-|z|)^{q\d-\b(\nu)}\widehat{\nu}(z)^{\frac{1}{q}}}\\
    &=\left\|S_{b,c}\left(|g|\widetilde{\omega}^\frac1p\right)\right\|_{L^q_\eta}^q,\quad g\in H^\infty,
    \end{split}
    \end{equation}
where $b=s-\d-\frac{2\b(\om)}{p}+\frac1q-\frac{\b(\nu)}{q}$, $c=2+s-2\d-\frac{2\b(\om)}{p}-\frac2p+\frac2q$ and $\eta(z)=\frac{\nu(z)(1-|z|)^{\b(\nu)-\d q}}{\widehat{\nu}(z)}$ for all $z\in\D$. Again we will appeal to Lemma~\ref{lemma:S} with $\sigma\equiv1$. First observe that $\eta\in\DDD$ and $\widehat{\eta}(r)\asymp (1-r)^{\beta(\nu)-\d q}$ for all $0\le r<1$ by Lemma~\ref{le:nuhiyperbolic}. Hence
    \begin{equation*}
    \begin{split}
    \int_r^1\frac{(1-t)^{c-2}}{\widehat{\eta}(t)^{\frac1q}}\,dt
    &\asymp\int_r^1(1-t)^{s-\d+\frac{2\beta(\om)}{p}-\frac2p+\frac2q-\frac{\beta(\nu)}{q}}\,dt
    \asymp(1-r)^{1+s-\d+\frac{2\beta(\om)}{p}-\frac2p+\frac2q-\frac{\beta(\nu)}{q}}\\
    &\asymp \frac{(1-r)^{c-1}}{\widehat{\eta}(r)^{\frac1q}},\quad 0\le r<1,
    \end{split}
    \end{equation*}
and, by Lemma~\ref{lemma:d-hat-new}(iii),
    \begin{equation*}
    \begin{split}
    \int_0^r \frac{\eta(t)}{(1-t)^{\frac{cq}{p}-1}\widehat{\eta}(t)^{1/p'}}\,dt
    &\asymp\int_0^r \frac{\nu(t)}{\widehat{\nu}(t)(1-t)^{\frac{q}{p}\left(2+s-2\d-\frac{2\b(\om)}{p}-\frac2p+\frac2q\right)-1-\frac{\beta(\nu)-q\d}{p}}} \, dt\\
    &\lesssim\frac{1}{(1-r)^{\frac{q}{p}\left(2+s-2\d-\frac{2\b(\om)}{p}-\frac2p+\frac2q\right)-1-\frac{\beta(\nu)-q\d}{p}}}
    \asymp\frac{\widehat{\eta}(r)^{\frac1p}}{(1-r)^{\frac{cq}{p}-1}},\quad 0\le r<1,
    \end{split}
    \end{equation*}
so the hypotheses of Lemma~\ref{lemma:S} are satisfied. Moreover,
    \begin{equation*}
    \begin{split}
    (1-r)^{2+b-c+\frac1q-\frac1p}\frac{\widehat{\eta}(r)^{\frac1q}}{\widehat{\sigma}(r)^{\frac1p}}
    \asymp1,\quad 0\le r<1,
    \end{split}
    \end{equation*}
and hence \eqref{I5} together with Lemmas~\ref{lemma:S} and \ref{LemmaDeMierda} imply $I_5(g)\lesssim\|g\|_{A^p_{\widetilde{\om}}}^q\asymp\|g\|_{A^p_{\om}}^q$ for all $g\in H^\infty$.
This finishes the proof of the proposition.
\end{proof}

In order to prove the necessity part of Theorem~\ref{th:Hankelqbiggerp} some definitions are needed. For $\eta>-1$ and a radial weight $\om$, let $b_{z,\omega}^\eta=B_z^\eta/\|B_z^\eta\|_{A^p_\omega}$ for $z\in\D$, where $B_z^\eta(\z)=(1-\bar{z}\z)^{-(2+\eta)}$. For each $f\in L^1_\nu$, define
    $$
    g_{z,\omega}^\eta(\zeta)=\frac{P_\nu(\overline{f}b_{z,\omega}^\eta)(\zeta)}{b_{z,\omega}^\eta(\zeta)},\quad \z\in\D,
    $$
and note that $g_{z,\omega}^\eta$ is a well-defined analytic function in $\D$ because the standard Bergman kernel $b_{z,\omega}^\eta$ has no zeros.
If $\nu,\omega$ are weights, $\eta>-1$ and $0<p,q<\infty$, let us consider the global mean oscillation
    $$
    \|fb_{z,\omega}^\eta-\overline{g_{z,\omega}^\eta (z)}b_{z,\omega}^\eta\|_{L^q_\nu},\quad z\in\D.
    $$

\begin{proposition}\label{pr:Hankelnecessity}
Let $1<p\le q<\infty$, $f\in L^q_\nu$, $\om\in\DD$,  $\nu\in B_q$ a radial weight and $\g(z)=\g_{\om,\nu,p,q}(z)=\frac{\widehat{\nu}(z)^\frac1q(1-|z|)^{\frac1q}}{\widehat{\om}(z)^\frac1p(1-|z|)^{\frac1p}}$. If $H_f^\nu,H_{\overline{f}}^\nu:A^p_\om\to L^q_\nu$ are bounded, then there exists $\eta_0=\eta_0(\nu,\om)>-1$ such that
    \begin{equation*}
    \begin{split}
    \sup_{z\in\D}\|fb_{z,\omega}^\eta-\overline{g_{z,\omega}^\eta (z)}b_{z,\omega}^\eta\|_{L^q_\nu}
    &\le\|H_f^{\nu}\|_{A^p_\om\to L^q_\nu}+\|P_\eta\|_{L^q_\nu\to L^q_\nu}\left(\|H_f^{\nu}\|_{A^p_\om\to L^q_\nu}
    +\|H_{\overline{f}}^\nu\|_{A^p_\om\to L^q_\nu}\right).
    \end{split}
    \end{equation*}
for each $\eta\ge \eta_0$. Moreover, there exists $r_0=r_0(\nu)>0$ such that for each fixed $r\ge r_0$ and $\eta\ge \eta_0$,
    \begin{equation*}
    \begin{split}
    \sup_{z\in\D}\|fb_{z,\omega}^\eta-\overline{g_{z,\omega}^\eta (z)}b_{z,\omega}^\eta\|_{L^q_\nu}
    &\gtrsim\sup_{z\in\D}
    \g(z)\left(
   \frac{1}{\nu(\Delta(z,r))} \int_{\Delta(z,r)}|f(\z)-\overline{g_{z,\omega}^\eta(z)}|^q\nu(\z)\,dA(\z)\right)^{\frac{1}{q}}.
    \end{split}
    \end{equation*}
\end{proposition}

\begin{proof}
The definition of the Hankel operator along with triangle inequality gives
    \begin{equation*}
    \begin{split}
    \|fb_{z,\omega}^\eta-\overline{g_{z,\omega}^\eta (z)}b_{z,\omega}^\eta\|_{L^q_\nu}
    &\le\|H_f^{\nu}(b_{z,\omega}^\eta)\|_{L^q_\nu}+\|P_\nu(fb_{z,\omega}^\eta)-\overline{g_{z,\omega}^\eta(z)}b_{z,\omega}^\eta\|_{L^q_\nu}\\
    &\le\|H_f^{\nu}\|_{A^p_\om\to L^q_\nu}\|b_{z,\omega}^\eta\|_{A^p_\om}+\|P_\nu(fb_{z,\omega}^\eta)-\overline{g_{z,\omega}^\eta(z)}b_{z,\omega}^\eta\|_{L^q_\nu}\\
    &=\|H_f^{\nu}\|_{A^p_\om\to L^q_\nu}+\|P_\nu(fb_{z,\omega}^\eta)-\overline{g_{z,\omega}^\eta(z)}b_{z,\omega}^\eta\|_{L^q_\nu}.
    \end{split}
    \end{equation*}
If $g\in A^1_\eta$, then the reproducing formula for the standard weighted Bergman projection yields $\overline{g(z)}b_{z,\omega}^\eta = P_\eta(\overline{g}b_{z,\omega}^\eta)$. Since $\nu\in B_q$ is radial and $f\in L^q_\nu$, we have $\nu\in\DDD$ and $P_\nu(fb^\eta_z)\in A^q_\nu$ by Proposition~\ref{pr:bqDDD}. Therefore $g_z^\eta\in A^q_\nu$ for all $z\in\D$. Moreover, $A^q_\nu\subset A^q_\eta\subset A^1_\eta$ if $\eta>\frac{\beta(\nu)}{q}-1$ by Lemma~\ref{Lemma:weights-in-D-hat}(ii). It follows that
    $$
    \|P_\nu(fb_{z,\omega}^\eta)-\overline{g_{z,\omega}^\eta(z)}b_{z,\omega}^\eta\|_{L^q_\nu}
    =\|P_\nu(fb_{z,\omega}^\eta)-P_\eta(\overline{g_{z,\omega}^\eta}b_{z,\omega}^\eta)\|_{L^q_\nu}=\|P_\eta(P_\nu(fb_{z,\omega}^\eta)-\overline{g_{z,\omega}^\eta}b_{z,\omega}^\eta)\|_{L^q_\nu},\quad z\in\D.
    $$
By \cite{BB}, there exists $\eta_1=\eta_1(\nu)>\frac{\beta(\nu)}{q}-1$ such that $P_\eta:L^q_\nu\to L^q_\nu$ is bounded if $\eta\ge\eta_1$. Therefore
    $$
    \|P_\nu(fb_{z,\omega}^\eta)-\overline{g_{z,\omega}^\eta(z)}b_{z,\omega}^\eta\|_{L^q_\nu}
    \le\|P_\eta\|_{L^q_\nu\to L^q_\nu}\|P_\nu(fb_{z,\omega}^\eta)-\overline{g_{z,\omega}^\eta}b_{z,\omega}^\eta\|_{L^q_\nu},\quad z\in\D,\quad \eta\ge \eta_1.
    $$
The triangle inequality yields
    \begin{equation*}
    \begin{split}
    \|P_\nu(fb_{z,\omega}^\eta)-\overline{g_{z,\omega}^\eta}b_{z,\omega}^\eta\|_{L^q_\nu}
    &\le\|fb_{z,\omega}^\eta-P_\nu(fb_{z,\omega}^\eta)\|_{L^q_\nu}+\|fb_{z,\omega}^\eta-\overline{g_{z,\omega}^\eta}b_{z,\omega}^\eta\|_{L^q_\nu}
   \\ &  =\|H_f^{\nu}(b_{z,\omega}^\eta)\|_{L^q_\nu}+\|\overline{f}b_{z,\omega}^\eta-g_{z,\omega}^\eta b_{z,\omega}^\eta\|_{L^q_\nu}\\
    &\le\|H_f^{\nu}\|_{A^p_\om\to L^q_\nu}\|b_{z,\omega}^\eta\|_{A^p_\om}
    +\|\overline{f}b_{z,\omega}^\eta-P_\nu(\overline{f}b_{z,\omega}^\eta)\|_{L^q_\nu}\\
    &=\|H_f^{\nu}\|_{A^p_\om\to L^q_\nu}
    +\|H_{\overline{f}}^\nu(b_{z,\omega}^\eta)\|_{L^q_\nu}
    \le\|H_f^{\nu}\|_{A^p_\om\to L^q_\nu}
    +\|H_{\overline{f}}^\nu\|_{A^p_\om\to L^q_\nu}.
    \end{split}
    \end{equation*}
By combining the above estimates we deduce
    \begin{equation*}
    \begin{split}
    \|fb_{z,\omega}^\eta-\overline{g_{z,\omega}^\eta (z)}b_{z,\omega}^\eta\|_{L^q_\nu}
    &\le\|H_f^{\nu}\|_{A^p_\om\to L^q_\nu}+\|P_\eta\|_{L^q_\nu\to L^q_\nu}\left(\|H_f^{\nu}\|_{A^p_\om\to L^q_\nu}
    +\|H_{\overline{f}}^\nu\|_{A^p_\om\to L^q_\nu}\right),
    \end{split}
    \end{equation*}
for any $\eta\ge \eta_1(\nu)$.

To see the second one, first observe that \cite[Corollary~2]{PR2016/1} and Lemma~\ref{Lemma:weights-in-D-hat}(ii) give
    \begin{equation*}
    \begin{split}
    \|B^\eta_z\|_{A^p_\om}^p
    &\asymp\int_0^{|z|}\frac{\widehat{\om}(t)}{(1-t)^{p(2+\eta)}}\,dt
    \lesssim\frac{\widehat{\om}(|z|)}{(1-|z|)^{\beta(\om)}} \int_0^{|z|}\frac{1}{(1-t)^{p(2+\eta)-\beta(\om)}}\,dt\\
    &\asymp\frac{\widehat{\om}(z)}{(1-|z|)^{p(2+\eta)-1}},\quad |z|\to1^-,
    \end{split}
    \end{equation*}
provided $\eta>\frac{\b(\om)+1}{p}-2$. Moreover, by \eqref{8} there exists $r_0=r_0(\nu)>0$ such that $(1-|z|)\widehat{\nu}(z)\asymp \nu(\Delta(z,r_0))$ for any $r\ge r_0$. Hence, for each $r\ge r_0$ we have
    \begin{equation*}
    \begin{split}
    \|fb_{z,\omega}^\eta-\overline{g_{z,\omega}^\eta (z)}b_{z,\omega}^\eta\|_{L^q_\nu}^q
    &\ge\int_{\Delta(z,r)}|f(\z)-\overline{g_{z,\omega}^\eta(z)}|^q|b^\eta_z(\z)|^q\nu(\z)\,dA(\z)\\
    &\asymp\frac{1}{\|B_z^\eta\|_{A^p_\om}^q(1-|z|)^{q(2+\eta)}}\int_{\Delta(z,r)}|f(\z)-\overline{g_{z,\omega}^\eta(z)}|^q\nu(\z)\,dA(\z)\\
    &\asymp\frac{1}{\widehat{\om}(z)^\frac{q}{p}(1-|z|)^{\frac{q}{p}}}\int_{\Delta(z,r)}|f(\z)-\overline{g_{z,\omega}^\eta(z)}|^q\nu(\z)\,dA(\z),
    \\ & \asymp \frac{\widehat{\nu}(z)(1-|z|)}{\widehat{\om}(z)^{\frac{q}{p}}(1-|z|)^{\frac{q}{p}}}
    \frac{1}{\nu(\Delta(z,r))}\int_{\Delta(z,r)}|f(\z)-\overline{g_{z,\omega}^\eta(z)}|^q\nu(\z)dA(\z).
    \end{split}
    \end{equation*}
The second claim for $\eta_0=\max\{\eta_1,\frac{\b(\om)+1}{p}-2\}$ is now proved.
\end{proof}

\begin{Prf}{\em{ of Theorem~\ref{th:Hankelqbiggerp}.}}
If $H_{f}^\nu,H_{\overline{f}}^\nu:A^p_\om\to L^q_\nu$ are bounded, then $f\in\BMO(\Delta)$ by Proposition~\ref{pr:Hankelnecessity} and Theorem~\ref{th:BMOdecom}.

Conversely, let $f\in\BMO(\Delta)$. Then $f$ can be decomposed as $f=f_1+f_2$, where $f_1\in\BA(\Delta)$ and $f_2\in\BO(\Delta)$, by Theorem~\ref{th:BMOdecom}(ii). Proposition~\ref{pr:SufBopart} shows that $H_{f_2}^\nu,H_{\overline{f_{2}}}^\nu:A^p_\om\to L^q_\nu$ are bounded.
Moreover, since $\nu\in B_q$ is radial, $\nu\in\DDD$ and $P_\nu: L^q_\nu\to L^q_\nu$ is bounded by Proposition~\ref{pr:bqDDD}.
Therefore Lemma~\ref{ba} yields
    $$
    \| H_{f_1}^\nu(g)\|^q_{L^q_\nu}
    \le\|f_1g\|_{L^q_\nu}+ \|P_\nu(f_1g)\|_{L^q_\nu}
    \lesssim\|f_1g\|_{L^q_\nu}
    \lesssim\|g\|_{A^p_\om}\quad g\in H^\infty.
    $$
It follows that $H_{f}^\nu,H_{\overline{f}}^\nu:A^p_\om\to L^q_\nu$ are bounded.
\end{Prf}

\section{Anti-analytic symbols}
\label{Anti-analytic symbols}

Recall that the space $\Bg$ consists of $f\in\H(\D)$ such that
    $$
    \|f\|_{\Bg}=\sup_{z\in\D}|f'(z)|(1-|z|)\gamma(z)+|f(0)|<\infty,
    $$
where $\g(z)=\frac{\widehat{\nu}(z)^\frac1q(1-|z|)^{\frac1q}}{\widehat{\om}(z)^\frac1p(1-|z|)^{\frac1p}}$ for all $z\in\D$.

\begin{proposition}\label{pr:ana}
Let $1<p\le q<\infty$, $\om,\nu\in\DDD$ and $r\ge r_0$, where $r_0=r_0(\nu)>0$ is that of Theorem~\ref{th:BMOdecom}(i). Then $\BMO(\Delta)\cap\H(\D)=\BMO(\Delta)_{\om,\nu,p,q,r}\cap\H(\D)=\Bg$.
\end{proposition}

\begin{proof}
Let first $f\in\Bg$. By Theorem~\ref{th:BMOdecom}(iv) to deduce $f\in\BMO(\Delta)$ it is enough to prove
    \begin{equation*}
    \sup_{z\in\D}\frac{(1-|z|)^{c+1}\gamma(z)^q}{\widehat{\nu}(z)}
    \int_\D|f(\z)-f(z)|^q\frac{(1-|\z|^2)^\sigma}{|1-z\overline{\z}|^{2+c+\sigma}}\nu(\z)\,dA(\z)<\infty
    \end{equation*}
for some  $\sigma>0$ and
    \begin{equation}
    \begin{split}\label{numbers}
    c>2\frac{q}{p}\left(\b(\om)+1\right)+\sigma+\max\left\{2\b(\nu),\gamma(\nu)\right\}.
    \end{split}
    \end{equation}
Since $f\in \H(\D)$, the function $(f(\z)-f(z))(1-\z\overline{z})^{-\frac{2+c+\sigma}{q}}$ is an analytic function in $\z$ for each $z\in\D$. Therefore Lemma~\ref{LemmaDeMierda} shows that \eqref{eq:hol1} is equivalent to
    \begin{equation}\label{eq:hol1}
    \sup_{z\in\D}\frac{(1-|z|)^{c+1}\gamma(z)^q}{\widehat{\nu}(z)}
    \int_\D |f(\z)-f(z)|^q\frac{(1-|\z|^2)^{\sigma-1}}{|1-z\overline{\z}|^{2+c+\sigma}}\widehat{\nu}(\z)\,dA(\z)<\infty.
    \end{equation}
Further, Lemma~\ref{Lemma:weights-in-D-hat}(ii) yields
    \begin{equation*}
    \begin{split}
    &\frac{(1-|z|)^{c+1}\gamma(z)^q}{\widehat{\nu}(z)}
    \int_\D |f(\z)-f(z)|^q\frac{(1-|\z|^2)^{\sigma-1}}{|1-z\overline{\z}|^{2+c+\sigma}}\widehat{\nu}(\z)\,dA(\z)\\
    &\lesssim
    (1-|z|)^{c+1}\gamma(z)^q\int_{\D\setminus D(0,|z|)}|f(\z)-f(z)|^q\frac{(1-|\z|^2)^{\sigma-1}}{|1-z\overline{\z}|^{2+c+\sigma}}\,dA(\z)\\
    &\quad+(1-|z|)^{c+1-\beta(\nu)}\gamma(z)^q
    \int_{D(0,|z|)}|f(\z)-f(z)|^q\frac{(1-|\z|^2)^{\sigma+\beta(\nu)-1}}{|1-z\overline{\z}|^{2+c+\sigma}}\,dA(\z)\\
    &\le(1-|z|)^{c+1}\gamma(z)^q\int_{\D}|f(\z)-f(z)|^q\frac{(1-|\z|^2)^{\sigma-1}}{|1-z\overline{\z}|^{2+c+\sigma}}\,dA(\z)\\
    &\quad+(1-|z|)^{c+1-\beta(\nu)}\gamma(z)^q\int_{\D}|f(\z)-f(z)|^q \frac{(1-|\z|^2)^{\sigma+\beta(\nu)-1}}{|1-z\overline{\z}|^{2+c+\sigma}}\,dA(\z)\\
    &=I_1(z)+I_2(z),\quad z\in\D.
    \end{split}
    \end{equation*}
Fix $\sigma>\max\left\{0,1-\frac{q}{p}(1+\a(\om)) +q\b(\nu) \right\}$ and $c$ satisfying \eqref{numbers}. Then
    $$
    c>\max\left \{ \b(\nu)-1, -2+\b(\nu)+\frac{q}{p}(1+\b(\om))-q\a(\nu)\right\}.
    $$
Therefore, \cite[Lemma~7]{PZMathAnn15} together with Lemmas~\ref{Lemma:weights-in-D-hat}(ii) and~\ref{Lemma:weights-in-D} gives
    \begin{equation*}
    \begin{split}
    I_1(z) &\lesssim (1-|z|)^{c+1}\gamma(z)^q \int_{\D} |f'(\z)|^q \frac{(1-|\z|^2)^{\sigma+q-1}}{|1-z\overline{\z}|^{2+c+\sigma}}\,dA(\z)\\
    &\lesssim\|f\|^q_{\Bg}(1-|z|)^{c+1}\gamma(z)^q \int_{\D} \gamma(\z)^{-q} \frac{(1-|\z|^2)^{\sigma-1}}{|1-z\overline{\z}|^{2+c+\sigma}}\,dA(\z)\\
    &\asymp\|f\|^q_{\Bg}(1-|z|)^{c+1}\gamma(z)^q
    \int_0^1\frac{\widehat{\om}(s)^{\frac{q}{p}}}{\widehat{\nu}(s)}\frac{(1-s)^{\frac{q}{p}+\sigma-2}}{(1-s|z|)^{1+\sigma+c}}\,ds\\
    &\lesssim\|f\|^q_{\Bg}(1-|z|)^{c+1}\gamma(z)^q\frac{\widehat{\om}(z)^{\frac{q}{p}}(1-|z|)^{\a(\nu)}}
    {(1-|z|)^{\frac{q}{p}\beta(\om)}\widehat{\nu}(z)}
    \int_0^{|z|}\frac{ds}{(1-s)^{3+c-\frac{q}{p}-\frac{q}{p}\beta(\om)+\a(\nu)}}\\
    &\quad+\|f\|^q_{\Bg}(1-|z|)^{-\sigma}\gamma(z)^q
    \frac{\widehat{\om}(z)^{\frac{q}{p}}(1-|z|)^{\b(\nu)}}{(1-|z|)^{\frac{q}{p}\a(\om)}\widehat{\nu}(z)}
    \int_{|z|}^{1}(1-s)^{\frac{q}{p}+\sigma-2+\frac{q}{p}\a(\om)-\b(\nu)}\,ds\\
    &\lesssim\|f\|^q_{\Bg}(1-|z|)^{-1+\frac{q}{p}}\gamma(z)^q\frac{\widehat{\om}(z)^{\frac{q}{p}}}{\widehat{\nu}(z)}
    \asymp\|f\|^q_{\Bg}<\infty,\quad z\in\D,
    \end{split}
    \end{equation*}
and
    \begin{equation*}
    \begin{split}
    I_2(z)
    &\lesssim(1-|z|)^{c+1-\b(\nu)}\gamma(z)^q\int_{\D}|f'(\z)|^q\frac{(1-|\z|^2)^{\sigma+\b(\nu)+q-1}}{|1-z\overline{\z}|^{2+c+\sigma}}\,dA(\z)\\
    &\lesssim\|f\|^q_{\Bg}(1-|z|)^{c+1-\b(\nu)}\gamma(z)^q
    \int_{\D}\gamma(\z)^{-q} \frac{(1-|\z|^2)^{\b(\nu)+\sigma-1}}{|1-z\overline{\z}|^{2+c+\sigma}}\,dA(\z)\\
    &\asymp\|f\|^q_{\Bg}(1-|z|)^{c+1-\b(\nu)}\gamma(z)^q
    \int_0^1\frac{\widehat{\om}(s)^{\frac{q}{p}}}{\widehat{\nu}(s)}\frac{(1-s)^{\b(\nu)+\frac{q}{p}+\sigma-2}}{(1-s|z|)^{1+\sigma+c}}\,ds\\
    &\lesssim\|f\|^q_{\Bg}(1-|z|)^{c+1-\b(\nu)}\gamma(z)^q\frac{\widehat{\om}(z)^{\frac{q}{p}}(1-|z|)^{\a(\nu)}}{(1-|z|)^{\frac{q}{p}\beta(\om)}\widehat{\nu}(z)^q}
    \int_0^{|z|}\frac{ds}{(1-s)^{3+c-\b(\nu)-\frac{q}{p}-\frac{q}{p}\beta(\om)+\a(\nu)}}\\
    &\quad+\|f\|^q_{\Bg}(1-|z|)^{-\sigma-\b(\nu)}\gamma(z)^q
    \frac{\widehat{\om}(z)^{\frac{q}{p}}(1-|z|)^{\b(\nu)}}{(1-|z|)^{\frac{q}{p}\a(\om)}\widehat{\nu}(z)}
    \int_{|z|}^{1}(1-s)^{\b(\nu)+\frac{q}{p}+\sigma-2+\frac{q}{p}\a(\om)-\b(\nu)}\,ds\\
    &\lesssim\|f\|^q_{\Bg}(1-|z|)^{-1+\frac{q}{p}}\gamma(z)^q\frac{\widehat{\om}(z)^{\frac{q}{p}}}{\widehat{\nu}(z)^q}
    \asymp\|f\|^q_{\Bg}<\infty,\quad z\in\D.
    \end{split}
    \end{equation*}
By combining these estimates we deduce $f\in\BMO(\Delta)$, and thus $\Bg\subset \H(\D)\cap \BMO(\Delta)$.

Assume now that $f\in\H(\D)\cap \BMO(\Delta)$. Then \eqref{eq:hol1} holds for some $\sigma>1$ and $c$ satisfying \eqref{numbers}. Therefore \eqref{8} implies
    \begin{equation*}
    \begin{split}
    \infty &>\sup\frac{(1-|z|)^{c+1}\gamma(z)^q}{\widehat{\nu}(z)}\int_\D |f(\z)-f(z)|^q \frac{(1-|\z|^2)^{\sigma-1}}{|1-z\overline{\z}|^{2+c+\sigma}}\widehat{\nu}(\z)dA(\z)\\
    &\gtrsim \frac{\gamma(z)^q}{(1-|z|)^2 \widehat{\nu}(z)}\int_{\Delta(z,r)} |f(\z)-f(z)|^q \widehat{\nu}(\z)\,dA(\z)\\
    &\asymp \frac{\gamma(z)^q}{|\Delta(z,r)|}\int_{\Delta(z,r)} |f(\z)-f(z)|^q \,dA(\z),\quad z\in\D.
    \end{split}
    \end{equation*}
By arguing as in \cite[1653--1654]{PZZ} we deduce $\H(\D)\cap\BMO(\Delta)\subset\Bg$.
\end{proof}

The space $\Bg$ consists of constant functions only if $\limsup_{|z|\to 1^-}((1-|z|)\gamma(|z|))^{-1}=0$.
Moreover, $\Bg$ is a subset of the disc algebra if $((1-x)\gamma(x))^{-1}\in L^1(0,1)$, and $\Bg$ coincides with a Bloch-type space if $\gamma$ is decreasing.

\medskip

\begin{Prf}{\em{ of Theorem~\ref{th:analytic}.}}
Since $f\in A^1_\nu$, the operator $H^{\nu}_{\overline{f}}$ is densely defined. If $H^{\nu}_{\overline{f}}:A^p_\om \to L^q_\nu$ is bounded, choosing $g\equiv1$ it follows that $f\in A^q_\nu$, and therefore $f\in\Bg$ by Theorem~\ref{th:Hankelqbiggerp} and Proposition~\ref{pr:ana}.

Conversely, assume $f\in\Bg$. Since $\nu\in B_q$ is radial, Proposition~\ref{pr:bqDDD} implies $\nu\in\DDD$. Therefore
Lemmas~\ref{Lemma:weights-in-D-hat}(ii) and~\ref{Lemma:weights-in-D} yield
    \begin{equation*}
    \begin{split}
    \|f\|^q_{A^q_\nu}
    &\lesssim\int_{\D}\left(\int_0^{|z|}\left|f'\left(s\frac{z}{|z|}\right)\right|\,ds \right)^q\nu(z)\,dA(z)+|f(0)|^q\\
    &\lesssim\|f\|^q_{\Bg}\left(1+\int_0^1\left(\int_0^t \frac{ds}{(1-s)\gamma(s)}\right)^q\nu(t)\,dt\right)\\
    &\lesssim\|f\|^q_{\Bg}\left(1+\int_0^1
    \left(\int_0^t\frac{\widehat{\om}(s)^{\frac1p}}{\widehat{\nu}(s)^{\frac1q}(1-s)^{1+\frac{1}{q}-\frac{1}{p}}}\,ds\right)^q\nu(t)\,dt\right)\\
    &\lesssim \|f\|^q_{\Bg}\left(1+\int_0^1\frac{\widehat{\om}(t)^{\frac{q}{p}}}{\widehat{\nu}(t)(1-t)^{\frac{q\a(\om)}{p}-\b(\nu)}}
    \left(\int_0^t \frac{ds}{(1-s)^{1+\frac{1+\b(\nu)}{q}-\frac{1+\a(\om)}{p}}}\right)^q\nu(t)\,dt\right)
    \end{split}
    \end{equation*}
for all $f\in\H(\D)$. If $\frac{1+\b(\nu)}{q}-\frac{1+\a(\om)}{p}>0$, Lemma~\ref{lemma:d-hat-new}(ii) gives
    \begin{equation*}
    \begin{split}
    \|f\|^q_{A^q_\nu}
    &\lesssim\|f\|^q_{\Bg}\left(1+\int_0^1\frac{\widehat{\om}(t)^{\frac{q}{p}}(1-t)^{\frac{q}{p}-1}}{\widehat{\nu}(t)}\nu(t)\,dt\right)\\
    &\lesssim\|f\|^q_{\Bg}\left(1+\widehat{\om}(0)^{\frac{q}{p}-1}\int_0^1\frac{\widehat{\om}(t)\nu(t)}{\widehat{\nu}(t)}\,dt\right)
    \lesssim\|f\|^q_{\Bg}.
    \end{split}
    \end{equation*}
If $\frac{1+\b(\nu)}{q}-\frac{1+\a(\om)}{p}=0$, then Lemmas~\ref{Lemma:weights-in-D} and ~\ref{lemma:d-hat-new}(ii) yield
    \begin{equation*}
    \begin{split}
    \|f\|^q_{A^q_\nu}
    &\lesssim\|f\|^q_{\Bg}\left(1+\int_0^1
    \frac{\widehat{\om}(t)^{\frac{q}{p}}(1-t)^{\frac{q}{p}-1}}{\widehat{\nu}(t)}\log\frac{e}{1-t}\nu(t)\,dt\right)\\
    &\lesssim\|f\|^q_{\Bg}\widehat{\om}(0)^{\frac{q}{p}}
    \int_0^1\frac{(1-t)^{\a(\om)\frac{q}{p}+\frac{q}{p}-1}}{\widehat{\nu}(t)}\log\frac{e}{1-t}\nu(t)\,dt\\
    &\lesssim\|f\|^q_{\Bg}\int_0^1\frac{(1-t)^{\a(\om)\frac{q}{2p}+\frac{q}{p}-1}}{\widehat{\nu}(t)}\nu(t)\,dt
    \lesssim\|f\|^q_{\Bg}.
    \end{split}
    \end{equation*}
Finally, if $\frac{1+\b(\nu)}{q}-\frac{1+\a(\om)}{p}<0$, then Lemma~\ref{lemma:d-hat-new}(ii) gives
    \begin{equation*}
    \begin{split}
    \|f\|^q_{A^q_\nu}
    &\lesssim \|f\|^q_{\Bg}\left(1+\int_0^1\frac{\widehat{\om}(t)^{\frac{q}{p}}(1-t)^{\frac{q}{p}-1}}{\widehat{\nu}(t)}\nu(t)\,dt\right)
    \lesssim\|f\|^q_{\Bg}.
    \end{split}
    \end{equation*}
Therefore $f\in A^q_\nu$, and thus $\Bg\subset A^q_\nu$. This together with Theorem~\ref{th:Hankelqbiggerp} and Proposition~\ref{pr:ana} finishes the proof.
\end{Prf}

\bibliographystyle{amsplain}

\begin{thebibliography}{00}

\bibitem{AFP88}         J.~Arazy, S.~D.~Fisher and J.~Peetre,  Hankel operators on weighted Bergman spaces,
                        Amer. J. Math. 110 (1988), no. 6, 989--1053.

\bibitem{Arrousi}       H.~Arrousi, Function and Operator theory on Large Bergman spaces,
                        PhD. Thesis, Univ. of Barcelona (2016).

\bibitem{AxlerDuke}     S.~Axler, The Bergman space, the Bloch space, and commutators of multiplication operators,
                        Duke Math. J. 53 (1986), no. 2, 315--332.

\bibitem{Bek}           D.~Bekoll\'e,
                        In\'egalit\'es \'a poids pour le projecteur de Bergman dans la boule unit\'e de
                        $C^n$,
                        [Weighted inequalities for the Bergman projection in the unit ball of $C^n$]
                        Studia Math. 71 (1981/82), no. 3, 305--323.

\bibitem{BB}            D.~Bekoll\'e and A.~Bonami,
                        In\'egalit\'es \'a poids pour le noyau de
                        Bergman,
                        (French) C. R. Acad. Sci. Paris Sér. A-B 286 (1978), no. 18, 775--778.


\bibitem{BBCZ}          D.~Bekoll\'e,  C.~Berger, L.~A.~Coburn and K.~Zhu,
                        BMO in the Bergman metric on bounded symmetric domains,
                        J. Funct. Anal. 93 (1990), no. 2, 310--350.

\bibitem{BCZ1}          C.~Berger, L.~A.~Coburn and K.~Zhu,
                        BMO on the Bergman spaces of the classical domains,
                        Bull. Amer. Math. Soc. (N.S.) 17 (1987), no. 1, 133--136.

\bibitem{BCZ2}          C.~Berger, L.~A.~Coburn and K.~Zhu,
                        Function theory on Cartan domains and the Berezin-Toeplitz symbol calculus,
                        Amer. J. Math. 110 (1988), no. 5, 921--953.

\bibitem{CP2}           O.~Constantin and J.~A.~Pel\'aez,
                        Boundedness of the Bergman projection on $L^p$-spaces with exponential weights,
                        Bull. Sci. Math. 139 (2015), 245--268.

\bibitem{D1}            M.~R. Dostani{\'c},
                        Unboundedness of the Bergman projections on $L^p$-spaces with exponential weights,
                        Proc. Edinb. Math. Soc. (2) 47 (2004), 111--117.

\bibitem{GalPauAASF}    P.~Galanopoulos and J.~Pau,
                        Hankel operators on large weighted Bergman spaces,
                        Ann. Acad. Sci. Fenn. Math. 37 (2012), no. 2, 635--648.

\bibitem{PZMathAnn15}   J.~Pau and R.~Zhao,
                        Weak factorization and Hankel forms for weighted Bergman spaces on the unit ball,
                        Math. Ann. 363 (2015), no. 1-2, 363--383.

\bibitem{PZZ}           J.~Pau, R.~Zhao and K.~Zhu,
                        Weighted BMO and Hankel operators between Bergman spaces,
                        Indiana Univ. Math. J. 65 (2016), no. 5, 1639--1673.

\bibitem{PelSum14}      J~A.~ Pel\'aez,
                        Small weighted Bergman spaces,
                        Proceedings of the summer school in complex and harmonic analysis, and related topics, (2016).

\bibitem{PR2014Memoirs} J.~A.~Pel\'aez and J.~R\"atty\"a,
                        Weighted Bergman spaces induced by rapidly increasing weights,
                        Mem.~Amer.~Math.~Soc.~227 (2014).


\bibitem{PR2015/2}      J.~A.~Pel\'aez and J.~R\"atty\"a,
                        Embedding theorems for Bergman spaces via harmonic analysis,
                        Math. Ann. 362 (2015), no. 1-2, 205--239.

\bibitem{PR2016/1}      J.~A.~Pel\'aez and J.~R\"atty\"a,
                        Two weight inequality for Bergman projection,
                        J. Math. Pures Appl. 105 (2016), 102--130.

\bibitem{PRS18}         J.~A.~Pel\'aez, K.~Sierra and J.~R\"atty\"a, Atomic decomposition and Carleson measures for weighted
                          Mixed norm spaces, https://arxiv.org/abs/1709.07239

\bibitem{Peralakernels} A.~Per\"al\"a,
                        Vanishing Bergman Kernels on the Disk,
                        J. Geom. Anal. 28 (2018), no. 2, 1716--1727.


\bibitem{ZhaoIEOT}      R.~Zhao,
                        Generalization of Schur's test and its application to a class of integral operators on the unit ball of $\mathbb{C}^n$,
                        Integral Equations Operator Theory 82 (2015), no. 4, 519--532.

\bibitem{ZZMF08}        R.~Zhao and K.~Zhu,  Theory of Bergman spaces in the unit ball of $\mathbb{C}^n$,
                        Mem. Soc. Math. Fr. (N.S.) no. 115 (2008), vi+103 pp. (2009).


\bibitem{ZhuTAMS87}     K.~Zhu,
                        VMO, ESV, and Toeplitz operators on the Bergman space,
                        Trans. Amer. Math. Soc. 302 (1987), no. 2, 617--646.

\bibitem{Zhu1992}       K. Zhu,
                        BMO and Hankel operators on Bergman spaces,
                        Pacific J. Math. 155 (1992), no. 2, 377--395.

\bibitem{Zhu}           K.~Zhu,
                        Operator Theory in Function Spaces,
                        Second Edition, Math. Surveys and Monographs, Vol. 138, American Mathematical
                        Society: Providence, Rhode Island, 2007.

\end{thebibliography}

\end{document}